\documentclass[11pt,a4paper]{amsart}

\usepackage{amsthm, amsfonts, amssymb, amsmath, latexsym, enumerate}
\usepackage[all]{xy} \usepackage[latin1]{inputenc}
\usepackage{array}
\usepackage{pifont}
\usepackage{hyperref}
\usepackage{cite}

\usepackage[left=2.8cm,top=2.60cm,right=2.8cm]{geometry}

\usepackage{mathrsfs}
\newcommand{\mathbold}{\mathbf}

\newtheorem{thm}{Theorem}[section] \newtheorem{lem}[thm]{Lemma}
 \newtheorem{prop}[thm]{Proposition}
\newtheorem{claim}[thm]{Claim}
\newtheorem*{thm*}{Theorem}

\theoremstyle{definition} \newtheorem{rmk}[thm]{Remark}

\newcommand{\OO}{\mathscr{O}}

\newcommand{\EExt}{\mathscr{E}xt}
\newcommand{\HHom}{\mathscr{H}om}

\newcommand{\sE}{\mathscr{E}}
\newcommand{\EE}{\sE}
\newcommand{\sH}{\mathscr{H}}
\newcommand{\sF}{\mathscr{F}}

\newcommand{\sP}{\mathscr{P}}

\newcommand{\cH}{\mathcal{H}}

\newcommand{\cK}{\mathcal{K}}
\newcommand{\cL}{\mathcal{L}}
\newcommand{\cF}{\mathcal{F}}
\newcommand{\cV}{\mathcal{V}}

\newcommand{\cN}{\mathcal{N}}

\newcommand{\UU}{\mathcal{U}}

\newcommand{\cI}{\mathcal{I}}

\newcommand{\Mo}{{\sf M}}

\DeclareMathOperator{\rk}{rk}
\DeclareMathOperator{\Hilb}{Hilb}
\DeclareMathOperator{\Ext}{Ext} 

\DeclareMathOperator{\Hom}{Hom} 
\DeclareMathOperator{\im}{Im} \DeclareMathOperator{\cok}{cok}
\DeclareMathOperator{\HH}{H} \DeclareMathOperator{\hh}{h}
\DeclareMathOperator{\ext}{ext}

\DeclareMathOperator{\Pic}{Pic}

\newcommand{\Z}{\mathbb Z} \newcommand{\C}{\mathbb C}
 
 \newcommand{\p}{\mathbb P}

\DeclareMathOperator{\D}{\mathbf{D^b}}

\newcommand{\RRHHom}{\mathbold{R}\mathcal{H}om}
\newcommand{\RR}{\mathbold{R}}

\newcommand{\ph}{\mathbf{\Phi}}
\newcommand{\phs}{\mathbf{\Phi^{*}}}
\newcommand{\phx}{\mathbf{\Phi^{!}}}
\newcommand{\pho}{\mathbf{\Phi_0}}

\newcommand{\phos}{\mathbf{\Phi^{*}_0}}
\newcommand{\thx}{\mathbf{\Theta}}

\DeclareMathOperator{\ts}{\otimes}
\newcommand{\rr}{\rightarrow}

\newcommand{\mono}{\hookrightarrow}
\newcommand{\epi}{\twoheadrightarrow}
\newcommand{\xr}{\xrightarrow}
\newcommand{\f}{F}

\SelectTips{cm}{} \newdir{ >}{{}*!/-5pt/@{>}}
\numberwithin{equation}{section}

\allowdisplaybreaks[1]
\begin{document}


\title[Sheaves on Fano threefolds of genus 9]
{Rank $2$ stable sheaves with odd determinant\\
on Fano threefolds of genus $9$.}

\author{Maria Chiara Brambilla}
\email{{\tt brambilla@dipmat.univpm.it}}
\address{Dipartimento di Scienze Matematiche,
  Università Politecnica delle Marche\\
  Via Brecce Bianche, I-60131 Ancona - Italia}
\urladdr{\tt{\url{http://www.dipmat.univpm.it/~brambilla}}}

\author{Daniele Faenzi}
\email{{\tt daniele.faenzi@univ-pau.fr}}
\address{Université de Pau et des Pays de l'Adour \\
  Av. de l'Université - BP 576 - 64012 PAU Cedex - France}
\urladdr{\tt{\url{http://www.univ-pau.fr/~faenzi/}}}



\thanks{The first author was partially supported by INDAM and MIUR.
The second author was partially supported by GRIFGA and MIUR. 
}

\begin{abstract}
  By the description due to Mukai and Iliev,
  a smooth prime Fano threefold $X$ of genus $9$ is  
  associated to a surface $\p(\cV)$, ruled over a smooth plane quartic
  $\Gamma$.
  We use Kuznetsov's integral functor to
  study rank-$2$ stable sheaves  on $X$  with odd determinant.
  For each $c_2 \geq 7$, we prove that a component of their moduli space $\Mo_X(2,1,c_2)$ is birational to a
  Brill-Noether locus of bundles on $\Gamma$ having enough sections when
  twisted by $\cV$.
  

  Moreover we prove that $\Mo_X(2,1,7)$ is isomorphic to 
  the blowing-up of the Picard variety $\mathrm{Pic}^{2}(\Gamma)$
  along 
  the curve parametrizing lines contained in $X$.
\end{abstract}

\maketitle

\section{Introduction}

Let $X$ be a smooth projective threefold, whose Picard group is generated by an
ample divisor $H_{X}$.
We consider Maruyama's coarse moduli scheme $\Mo_{X}(r,c_{1},c_{2})$
of $H_{X}$-semistable rank $r$ sheaves $F$ on $X$ with $c_{i}(F)=c_{i}$ and $c_3(F)=0$.

Little is known about this space in general,
but many results are available in special cases.
For instance, rank $2$ bundles on
$\p^{3}$ have been intensively studied since
\cite{barth:some-properties}.

Since \cite{adhm} and \cite{atiyah-ward}, the case which has
attracted most attention is that of {\it instanton bundles}, i.e.\
stable rank $2$ bundles $F$ with $c_{1}(F)=0$,
$\HH^{2}(\p^{3},F(-2))=0$. Their moduli space is known to be
smooth and irreducible for $c_{2}(F) \leq 5$, see
\cite{kastylo-ottaviani}, \cite{coanda-tikhomirov-trautmann} and
references therein. The starting points in the investigation of
this case are Beilinson's theorem and the notion of monad, see
\cite{barth-hulek}, \cite{okonek-schneider-spindler}.

Now, if one desires to set up a similar analysis over a threefold
$X$ other than $\p^{3}$, one direction is to look at Fano
threefolds. Recall that if the anticanonical divisor $-K_{X}$ is
linearly equivalent to $i_{X} H_{X}$, for some positive integer
$i_{X}$, then the variety $X$ is called a {\it Fano threefold} of
index $i_{X}$. These varieties are in fact completely classified
by Iskovskih and later by Mukai, see \cite{fano-encyclo} and
references therein.

Our aim is to study the moduli space $\Mo_{X}(2,c_{1},c_{2})$ on a
Fano threefold $X$ of index $i_{X}=1$.
Recall that the genus of a Fano threefold $X$ of index $1$ is defined as $g=H_{X}^{3}/2+1$.
Notice that, since the rank of a sheaf $F$ in $\Mo_{X}(2,c_{1},c_{2})$
is $2$, one can assume $c_1\in\{0,1\}$. Accordingly, we speak of
bundles with odd or even determinant.

Let us focus on the case of odd determinant.
One sees that $\Mo_{X}(2,1,c_{2})$ is empty for
$c_{2} < m_{g}=\lceil g/2 \rceil+1$.
The case of minimal $c_{2}=m_{g}$ is well understood (see for instance
\cite{iliev-markushevich:genus-7} for genus $7$,
\cite{iliev-ranestad} for genus $9$,
\cite{kuznetsov:V22} for genus $12$).
For higher $c_{2}$, we are aware of the results contained in
\cite{iliev-markushevich:sing-theta:asian}, \cite{enrique-dani:v22},
\cite{iliev-manivel:genus-8}, \cite{brambilla-faenzi:genus-7}, where
only the last two papers study also the boundary of $\Mo_{X}(2,1,c_{2})$.

This paper, together with \cite{brambilla-faenzi:genus-7},  is devoted to the study of the space
$\Mo_{X}(2,1,c_{2})$ for $c_{2} > m_{g}$,
with a special emphasis  on $c_{2}=m_{g}+1$.
Our main idea is to make use of Kuznetsov's semiorthogonal decomposition of
the derived category of $X$ (see \cite{kuznetsov:hyperplane}), to
develop a suitable homological method, thus rephrasing the language of
monads and Beilinson's theorem.

More precisely, in this paper we focus on Fano threefolds $X$ of genus $9$.
Recall that, by a result of Mukai, \cite{mukai:curves-K3},
\cite{mukai:biregular}, the variety $X$ is a linear section of
the Lagrangian Grassmannian sixfold $\Sigma$.
We consider the orthogonal plane quartic $\Gamma$, and the integral
functor $\phx: \D(X) \to \D(\Gamma)$, according to Kuznetsov's theorem, \cite{kuznetsov:hyperplane}.
The functor is right adjoint to the fully faithful
functor $\ph$, provided by the universal sheaf $\EE$ on
$X \times \Gamma$ for the fine moduli space $\Gamma \cong \Mo_{X}(2,1,6)$.
Recall that the threefold $X$ is associated to a rank $2$ stable bundle $\cV$ on
$\Gamma$, in such a way that $\p(\cV)$ is isomorphic to the Hilbert
scheme $\sH^{0}_{2}(X)$ of conics contained in $X$, see \cite{iliev:sp3}.

For any $d\geq 7$, we proved in \cite{brambilla-faenzi:genus-7} that
there exists a component $\Mo(d)$ of
$\Mo_{X}(2,1,d)$, whose general element is a vector bundle $F$ with
$\HH^{k}(X,F(-1))=0$, for all $k$. Here we investigate in details the
properties of $\Mo(d)$.

The main result of this paper is the following.
\begin{thm*}
 The map $\varphi : F \mapsto \phx(F)$ gives:
 \begin{enumerate}[A)]
 \item \label{A} for any $d \geq 8$, a birational map
   of $\Mo(d)$ to a generically smooth $(2d-11)$-dimensional
   component of the Brill-Noether locus:
    \[
    \{ \cF \in \Mo_{\Gamma}(d-6,d-5) \, \lvert \,
    \hh^{0}({\Gamma},\cV \ts \cF) \geq d-6 \};
    \]
  \item \label{B} an isomorphism of $\Mo_{X}(2,1,7)$ with the blowing up of
    $\Pic^{2}(\Gamma)$ along a curve isomorphic to the Hilbert scheme $\sH^{0}_{1}(X)$
    of lines contained in $X$.
    The exceptional divisor consists of the sheaves in $\Mo_X(2,1,7)$
    which are not globally generated.
 \end{enumerate}
\end{thm*}

In particular we prove that $\Mo_{X}(2,1,7)$ is an irreducible
threefold which is smooth as soon as $\sH^{0}_{1}(X)$ is smooth.
Note that this result closely resembles that of
\cite{druel:cubic-3-fold}, \cite{iliev-markushevich:cubic},
\cite{markushevich-tikhomirov}, regarding rank $2$ sheaves on a smooth
cubic threefold in $\p^4$, and relying on the
Abel-Jacobi mapping.

The paper is organized as follows.
In the following section we set up some notation.
Then, in Section \ref{sec:genus-9}, we review the geometry of prime
Fano threefolds $X$ of genus $9$, and we interpret some well-known facts
concerning lines and conics contained in $X$ in the language of vector bundles.
In Section \ref{sec:stable}, we state and prove part \eqref{A} of the theorem above.
Section \ref{sec:7} is devoted to part \eqref{B}.

\section{Definitions and preliminary results}

Given a smooth complex projective $n$-dimensional polarized variety
$(X,H_X)$ and a sheaf $F$ on $X$, we write $F(t)$ for $F\ts \OO_X(t H_X)$.
Given a subscheme $Z$ of $X$, we write $F_Z$ for $F\ts \OO_Z$ and
we denote by $\cI_{Z,X}$ the ideal sheaf of $Z$ in $X$, and by
$N_{Z,X}$ its normal sheaf. We will frequently drop
the second subscript.
Given a pair of sheaves $(F,E)$ on $X$, we will write $\ext_X^k(F,E)$
for the dimension of the
group $\Ext_X^k(F,E)$, and
similarly $\hh^k(X,F) = \dim \HH^k(X,F)$.
The Euler characteristic of $(F,E)$
is defined as $\chi(F,E)=\sum_k (-1)^k \ext_X^k(F,E)$ and
$\chi(F)$ is defined as $\chi(\OO_X,F)$.
We denote by $p(F,t)$ the Hilbert polynomial $\chi(F(t))$ of the sheaf
$F$.
The degree $\deg(L)$ of a divisor class $L$ is defined as
the degree of $L \cdot H_X^{n-1}$. The degree of a sheaf $F$ is
defined as $\deg(c_1(F))$. The dualizing sheaf of $X$ is
denoted by $\omega_{X}$.
Given two vector bundles $E,F$, we denote by $e_{E,F}$ the natural evaluation 
map 
\[
e_{E,F}:\Hom_X(E,F)\ts E \to F.
\]

If $X$ is a smooth $n$-dimensional subvariety of $\p^m$, whose coordinate ring is
Cohen-Macaulay, then $X$ is said to be arithmetically Cohen-Macaulay
(ACM).
A locally free sheaf $F$ on an ACM variety $X$ is called {\em ACM}
({\em arithmetically Cohen-Macaulay}) if it has no intermediate
cohomology, i.e. if $\HH^k(X,F(t))=0$ for all integer $t$ and for any
$0<k<n$. The corresponding module over the coordinate ring of $X$ is
thus a maximal Cohen-Macaulay module.

Let us now recall a few well-known facts about semistable sheaves on projective varieties.
We refer to the book \cite{huybrechts-lehn:moduli} for a more detailed account of these notions.
We recall that a torsionfree coherent sheaf $F$ on $X$ is (Gieseker) {\it semistable} if for
any coherent subsheaf $E$, with $0<\rk(E)<\rk(F)$,
one has $p(E,t)/\rk(E) \leq p(F,t)/\rk(F)$ for $t\gg 0$.
The sheaf $F$ is called {\it stable} if the inequality above is always strict.

The {\it slope} of a sheaf $F$ of positive rank is defined as
$\mu(F) = \deg(F)/\rk(F)$.
We recall that a torsionfree coherent sheaf $F$ is  {\it $\mu$-semistable} if for
any coherent subsheaf $E$, with $0<\rk(E)<\rk(F)$,
one has $\mu(E) < \mu(F)$.
The sheaf $F$ is called {\it $\mu$-stable} if the above inequality is always strict.
We recall that the {\it discriminant} of a sheaf $F$ is
$\Delta(F) = 2 r c_2(F) - (r-1) c_1(F)^2$, 
where $c_k(F) \in \HH^{k,k}(X)$ is the $k$-th Chern class of $F$.
Bogomolov's inequality, see for instance \cite[Theorem 3.4.1]{huybrechts-lehn:moduli},
states that if $F$ is also $\mu$-semistable, then we have:
\begin{equation}
\label{eq:bogomolov}
\Delta(F)\cdot H_X^{n-2}\geq 0.
\end{equation}


We introduce here some notation concerning moduli spaces.
We denote by $\Mo_X(r,c_1,\ldots,c_n)$ the moduli space of
$S$-equivalence classes of rank $r$
torsionfree semistable sheaves on $X$ with Chern classes $c_1,\ldots,c_n$.
The Chern class $c_k$ will be
denoted by an integer as soon as $\HH^{k,k}(X)$ has dimension $1$.
We will drop the last values of the classes $c_k$ when they are
zero.
The moduli space of $\mu$-semistable sheaves is denoted by
$\Mo^{\mu}_X(r,c_1,\ldots,c_n)$.

Let us review some notation concerning the Hilbert scheme.
Given a numerical polynomial $p(t)$, we let $\Hilb_{p(t)}(X)$ be the
{\it Hilbert scheme} of closed subschemes of $X$ with Hilbert polynomial
$p(t)$.
In case $p(t)$ has degree one, we let $\sH^g_d(X)$ be the union of
components of $\Hilb_{p(t)}(X)$ containing integral
curves of degree $d$ and arithmetic genus $g$.

As a basic technical tool, we will use the bounded derived category.
Namely, given a smooth complex projective variety $X$, we will
consider the derived category $\D(X)$ of complexes of sheaves on $X$
with bounded coherent cohomology.
For definitions and notation we refer to
\cite{gelfand-manin:homological} and \cite{weibel:homological}.
In particular we write $[j]$ for the $j$-th shift to the right in the
derived category.


\smallskip

Let now $X$ be a smooth projective variety of dimension $3$.
Recall that $X$ is called {\it Fano} if its anticanonical divisor
class $-K_X$ is ample.
A Fano threefold  $X$ is {\it prime} if its Picard group is generated by the class of
$K_X$. These varieties are classified up to deformation, see for
instance \cite[Chapter IV]{fano-encyclo}. The number of
deformation classes is $10$, and they are characterized by the
{\it genus}, which is the integer $g$ such that
$\deg(X)=-K_X^3=2\, g-2$. Recall that the genus of a prime Fano
threefold take values in $\{2,\ldots,10,12\}$.

If $X$ is a prime Fano threefold of genus $g$, the Hilbert scheme
$\sH^{0}_{1}(X)$ of lines contained in $X$ is a scheme of pure
dimension $1$.
The threefold $X$ is said to be {\em exotic} if the Hilbert scheme
$\sH^{0}_{1}(X)$ contains a component which is nonreduced at any point.
It turns out that no threefold of genus $9$ is exotic, see
\cite{gruson-laytimi-nagaraj}. In particular the
normal bundle of a general line $L\subset X$ splits
as $\OO_{L} \oplus \OO_{L}(-1)$. 
It is well-known that, if $X$ is general, then the scheme
$\sH^{0}_{1}(X)$ is a smooth irreducible curve. 

Recall also that a smooth projective surface $S$ is a {\it K3 surface} if
it has trivial canonical bundle and irregularity zero.

Remark that the cohomology groups $\HH^{k,k}(X)$ of a prime Fano
threefold $X$ of genus $g$ are generated by the
divisor class $H_X$ (for $k=1$), the class $L_X$ of a line contained
in $X$ (for $k=2$), the class $P_X$ of a closed point of $X$ (for $k=3$).
Hence we will denote the Chern classes of a sheaf on $X$ by the integral
multiple of the corresponding generator. Recall that $H_X^2=(2\,g-2)L_X$.
We use an analogous notation on
a K3 surface $S$ of genus $g$.

We recall by \cite[Part II, Chapter 6]{huybrechts-lehn:moduli} that,
given a stable sheaf $F$ of rank $r$ on a K3 surface $S$ of sectional genus
$g$, with Chern classes $c_1,c_2$,
the dimension at $[F]$ of the moduli space $\Mo_S(r,c_1,c_2)$ is:
\begin{equation}
  \label{eq:dimension}
  \Delta(F) - 2\, (r^2-1).
\end{equation}

We recall finally the formula of Hirzebruch-Riemann-Roch, in the case
of prime Fano threefolds of genus $9$. Let $F$
be a rank $r$ sheaf on a prime Fano threefold $X$ of
genus $9$ with Chern classes $c_1,c_2,c_3$.
Then we have:
\begin{align*}
\chi(F) & = r + \frac{10}{3}\,c_1 +4\, c_1^2
-\frac{1}{2} c_2 +  \frac{8}{3}\,c_1^3
-\frac{1}{2}\,c_1\,c_2+\frac{1}{2}\,c_3.
\end{align*}

\section{Geometry of  prime Fano threefolds of genus 9} \label{sec:genus-9}

Throughout the paper we will denote by $X$ a smooth prime Fano threefold of
genus $9$.
In this section, we briefly sketch some of the basic features of $X$.
For a detailed account on the geometry of these varieties, and the
related ${\sf Sp}(3)$-geometry, we refer to the papers \cite{mukai:curves-K3},
\cite{mukai:biregular}, \cite{iliev:sp3}, \cite{iliev-ranestad}.
The divisor class $H_X$ embeds $X$ in $\p^{10}$ as an ACM variety.
It is well known that a general hyperplane section $S$ of $X$ is a smooth
K3 surface polarized by the restriction $H_S$ of $H_X$ to $S$,
with Picard number $1$ and sectional genus $9$.

By a result of Mukai, the threefold $X$ is isomorphic to a
$3$-codimensional linear section of the
Lagrangian Grassmannian $\Sigma$ of $3$-dimensional subspaces of a
$6$-dimensional vector space $V$ which are isotropic with respect to a
skew-symmetric $2$-form $\omega$.
The manifold $\Sigma$ is homogeneous for the complex Lie group ${\sf{Sp}}(3)$,
which acts on $V$ preserving $\omega$.
The Lie algebra of this group has dimension $21$, its Dynkin diagram
is of type ${\sf C}_{3}$ and the manifold $\Sigma$ is
${\sf{Sp}}(3)/{\sf P}(\alpha_{3})$.
In fact, $\Sigma$ is a Hermitian symmetric space.
It is equipped with a universal homogeneous rank $3$ subbundle $\UU$,
and we still denote by $\UU$ its restriction to $X$.
We have the universal exact sequence:
\begin{equation}
  \label{eq:universal-9}
  0 \to \UU \to V\ts \OO_X \to \UU^* \to 0.
\end{equation}

Let us review the properties of the vector bundle $\UU$. Its Chern
classes satisfy $c_1(\UU)=-1$, $c_2(\UU)=8$, $c_3(\UU)=-2$. The
bundle $\UU$ is exceptional by \cite{kuznetsov:hyperplane}.
Moreover, we have the following lemma.

\begin{lem} \label{stabU}
  The bundle $\UU$ is stable and ACM. The same is true for its
  restriction $\UU_{S}$ to a smooth hyperplane section surface $S$ with
  $\Pic(S)=\langle H_{S} \rangle$.
\end{lem}

\begin{proof}
Consider the Koszul complex:
\[
0 \to \wedge^{3} B \ts \OO_{\Sigma}(-3) \to \cdots \to B \ts
\OO_{\Sigma}(-1) \to \OO_{\Sigma} \to \OO_{X} \to 0,
\]
and tensor it with $\UU$.

By Bott's theorem  we know that, for any integer $t$, the
homogeneous vector bundles $\UU(t)$ on $\Sigma$ have natural
cohomology. Using Riemann-Roch's formula on $\Sigma$, we get
$\chi(\UU(-t))=0$, for $0\le t\le 3$.
We obtain:
\[
 \HH^k(\Sigma,\UU(-t))=0, \qquad \mbox{for} \qquad
 \left\{
 \begin{array}{l}
   \mbox{all $k$ and $0\le t\le 3$,} \\
   \mbox{$k \neq 0$ and $t\le-1$,} \\
   \mbox{$k \neq 6$ and $t\ge4$.}
 \end{array}
 \right.
\]

It easily follows that $\UU$ is ACM on $X$. Since $\wedge^2 \UU \cong
\UU^*(-1)$, by Serre duality we get $\HH^0(X,\wedge^2 \UU)=0$, so
$\UU$ is stable by Hoppe's criterion,
see \cite[Lemma 2.6]{hoppe:rang-4}, or
\cite[Theorem 1.2]{ancona-ottaviani:special}.

To check the statement on $S$, consider the defining exact sequence:
\begin{equation}
  \label{eq:sectionS}
  0 \to \OO_{X}(-1) \to \OO_{X} \to \OO_{S} \to 0.
\end{equation}

Since $\UU$ is ACM on $X$,
tensoring \eqref{eq:sectionS} by $\UU(-t)$, and using
$\HH^{0}(X,\UU)=0$, we get:
\begin{align*}
& \HH^{1}(S,\UU(t))=0,  && \mbox{for $t\geq 0$}, && \mbox{and}\quad \HH^{0}(S,\UU)=0.
\intertext{Tensoring \eqref{eq:sectionS} by $\UU^{*}(-t)$, recalling that we have proved
$\HH^{0}(X,\UU^{*}(-1))=0$, and that $\UU$ is ACM on $X$, making use of Serre duality we obtain:}
& \HH^{1}(S,\UU(t))=0, && \mbox{for $t\geq 1$}, && \mbox{and}\quad \HH^{0}(S,\UU^{*}(-1))=0.
\end{align*}

This proves that the bundle $\UU_{S}$ is ACM and that it is stable
again by Hoppe's criterion.
\end{proof}

\subsection{Universal bundles and the decomposition of the derived category}

Here we review the structure of the derived category of a smooth prime
Fano
threefold $X$ of genus $9$, in terms of the semiorthogonal
decomposition provided by \cite{kuznetsov:hyperplane}. We will
need to interpret this decomposition in terms of the universal
vector bundle of the moduli space $\Mo_{X}(2,1,6)$. In view of the
results of \cite{iliev-ranestad}, and recalling \cite[Lemma
3.4]{brambilla-faenzi:ACM}, the moduli space $\Mo_X(2,1,6)$ is
fine and isomorphic to a smooth plane quartic curve $\Gamma$. This
curve can be obtained as an orthogonal linear section of $\Sigma$ and
is also called the homologically projectively dual curve to $X$.
Let us denote by $\EE$ the universal vector bundle for the moduli
space $\Mo_X(2,1,6)= \Gamma$.
It is defined on $X \times \Gamma$, and we denote by $p$ and $q$ respectively the projections to $X$
and $\Gamma$.

We have the integral functor $\ph$ associated to $\EE$, and its
right and left adjoint functors $\phx$ and $\phs$, which are
defined by the formulas:
\begin{align}
& \label{phi} \ph :  \D(\Gamma) \rr \D(X), && \ph(-) = \RR p_*(q^*(-) \ts \EE), \\
& \label{phi!} \phx : \D(X) \rr \D(\Gamma), && \phx(-) = \RR q_*(p^*(-) \ts \EE^* (\omega_{\Gamma}))[1], \\
& \label{phi*} \phs : \D(X) \rr \D(\Gamma), && \phs(-) = \RR q_*(p^*(-) \ts \EE^* (-H_X))[3].
\end{align}

The topological invariants of $\EE$ are the following:
\[
c_1(\EE)=H_X+N, \qquad c_2(\EE)=6 L_X + H_X M+\eta,
\]
where $N$ and $M$ are divisor classes on $\Gamma$, and $\eta$ sits
in $\HH^{3}(X,\C)\ts \HH^{1}(\Gamma,\C)$.

\begin{lem}
  We have $\eta^{2}=6$ and $\deg(N)=2\deg(M)-1$.
\end{lem}

\begin{proof}
  Recall that $\EE$ is the universal bundle for $\Mo_{X}(2,1,6)$,
  and write $\EE_{y}$ for the bundle on $X$ corresponding to
  the point $y\in \Gamma$.
  By \cite[Lemma 3.3]{brambilla-faenzi:ACM}, we have
  $\Ext^2_{X}(\EE_{y},\EE_{z})=0$, for all $y,z \in \Gamma$.
  Moreover $\Ext^3_{X}(\EE_{y},\EE_{z})=0$, for all $y,z \in \Gamma$, 
by Serre duality and stability. 
Thus by Riemann-Roch formula it easily follows:
  \begin{align*}
    & \hom_{X}(\EE_{y},\EE_{y}) = \ext^{1}_{X}(\EE_{y},\EE_{y}) = 1, &&
    \mbox{for all $y\in \Gamma$,} \\
    & \hom_{X}(\EE_{y},\EE_{z}) = \ext^{1}_{X}(\EE_{y},\EE_{z}) = 0,
    && \mbox{for all $y\neq z \in \Gamma$.}
  \end{align*}

  This gives $\phx(\EE_{y}) \cong \OO_{y}$.
  By \cite[Proposition 3.5]{brambilla-faenzi:genus-7}, for any $y \in \Gamma$, the
  bundle $\EE_{y}$ satisfies:
  \[
  \HH^{k}(X,\EE^{*}_{y}) = 0, \qquad \mbox{for all $k \in \Z$},
  \]
  hence we have $\phx(\OO_{X})=0$.
  Plugging the equations $\chi(\phx(\OO_{X}))=0$ and
  $\chi(\phx(\EE_{y}))=1$
  into Grothendieck-Riemann-Roch's formula,
  we get our claim.
\end{proof}

By Kuznetsov's theorem, \cite{kuznetsov:hyperplane}, we have the
semiorthogonal decomposition:
\[
\D(X) = \langle \OO_X, \UU^*, \thx(\D(\Gamma)) \rangle,
\]
where $\thx$ is the integral functor associated to a sheaf
$\sF$ on $X \times \Gamma$, flat over $\Gamma$.
We would like to see that $\thx$ actually agrees with $\ph$.
We do this in a rather indirect way, in the following lemma.

\begin{lem}
  The sheaf $\sF$ is isomorphic to (a twist) of $\EE$.
\end{lem}

\begin{proof}
  It follows by \cite[Appendix A]{kuznetsov:hyperplane} that $\sF_{y}$
  fits into a long exact sequence:
  \[
  0 \to \OO_{X} \to \UU^{*} \to \sF_{y} \to \OO_{Z} \to 0,
  \]
  where $Z$ is the intersection of a 3-dimensional quadric contained in
  $\Sigma$ with a codimension $2$ linear section of $\Sigma$.
  Note that $\sF_{y}$ is torsionfree of rank $2$.
  Since $X$ does not contain planes or $2$-dimensional quadrics, $Z$ must be a conic.
  Therefore, we have $c_1(\OO_Z)=0$, $c_2(\OO_Z)=-2$, $c_3(\OO_Z)=0$.
  Thus we calculate $c_1(\sF_y)=1$, $c_2(\sF_y)=6$, $c_3(\sF_y)=0$, and we easily check that $\sF_y$
  is a stable sheaf, i.e. $\sF_y$ sits in $\Mo_{X}(2,1,6)$. Note that, by \cite[Proposition 3.5]{brambilla-faenzi:genus-7},
   $\sF_{y}$ must be a vector bundle. Since $\EE$ is a universal vector bundle for the fine moduli space
   $\Gamma = \Mo_X(2,1,6)$, we have thus that $\sF$ is the twist by a line bundle on $\Gamma$ of a pull-back of $\EE$ via a map
   $f: \Gamma \to \Gamma$.

  Note that if $f$ is not constant, then it is an isomorphism and we are done.
  Now, in view of \cite{bridgeland:fourier-mukai}, it is easy to prove that $f$ is not constant
  since $\thx$ is fully faithful.
  Indeed, the sheaf $\sF$ must satisfy:
  \begin{align}
    \nonumber & \Ext^{k}_{X}(\sF_{y},\sF_{z}) = 0,  && \mbox{for all $k$ if $y\neq z \in \Gamma$.}
  \end{align}
 But if $f$ was constant, we would have $\hom_{X}(\sF_{y},\sF_{z}) = 1$, for any $y, z \in \Gamma$.
\end{proof}

The semiorthogonal decomposition of $\D(X)$ can be thus rewritten as:
\begin{equation}
  \label{eq:DX}
  \D(X) = \langle \OO_X, \UU^*, \ph(\D(\Gamma)) \rangle.
\end{equation}

Then, given a sheaf $F$ over $X$, we have a functorial exact triangle:
\begin{equation}\label{triangolo}
\ph(\phx(F)) \to F \to \mathbf{\Psi} (\mathbf{\Psi^*}(F)),
\end{equation}
where $\mathbf{\Psi}$ is the inclusion of the subcategory
$\langle \OO_X, \UU^*_+ \rangle$ in $\D(X)$ and $\mathbf{\Psi^*}$ is the left
adjoint functor to $\mathbf{\Psi}$.
The $k$-th term of the complex $\mathbf{\Psi} (\mathbf{\Psi^*}(F))$ can be
written as follows:
\begin{equation}
  \label{quadrato}
  (\mathbf{\Psi} (\mathbf{\Psi^*}(F)))^k \cong
  \Ext_X^{-k}(F,\OO_X)^*\ts \OO_X\oplus \Ext_X^{1-k}(F,\UU_+)^*\ts \UU_+^*.
\end{equation}

\begin{rmk} \label{convenzione}
  The universal bundle $\EE$ is determined up to twisting by the
  pull-back of a line bundle on $\Gamma$.
  In order to simplify some computations, we adopt the convention:
  \[
  \deg(N)=\deg(\EE_{x})=5.
  \]
\end{rmk}

\begin{rmk}
Making use of mutations, one can easily write down the following
semiorthogonal decomposition of $\D(X)$:
\begin{equation}
  \label{eq:DX'}
  \D(X) = \langle  \pho(\D(\Gamma)) , \UU , \OO_X \rangle,
\end{equation}
where $\pho : \D(\Gamma) \rr \D(X)$ is defined as $\pho = \RR p_*(q^*(-)
\ts \EE(-H_{X}))$. Let $\phos$ be the left adjoint of the functor $\pho$.

Let $q_{1}$ and $q_{2}$ be the projections of $X \times X$ onto the
two factors, and denote by $\mathbf{U}$ the complex on $X \times X$
defined by the natural map $\UU \boxtimes \UU \to \OO_{X \times X}$,
where $\OO_{X \times X}$ has cohomological degree $0$.
Then the projection onto the subcategory $\langle \UU , \OO_X \rangle$
is given by the functor $\RR q_{2*} (q_{1}^{*}(-) \ts \mathbf{U})$.
\end{rmk}

\begin{lem} \label{pallone}
  We have the natural isomorphisms:
  \begin{align*}
    & \cH^{0}(\ph(\phs(\UU^{*}))) \cong \UU^{*}, \\
    & \cH^{1}(\ph(\phs(\UU^{*}))) \cong \UU(1).
  \end{align*}
\end{lem}

\begin{proof}
  Note that, for any object $F$ of $\D(X)$, we have $\phos(F(-1)) \cong
  \phs(F)$, and for any object $\cF$ of $\D(\Gamma)$, we have
  $\ph(\cF)(-1)\cong \pho(\cF)$.
  In particular, we get a natural isomorphism
  $\pho(\phos(\UU^{*}(-1)))(1) \cong \ph(\phs(\UU^{*}))$.

  By the decomposition \eqref{eq:DX'}, we get a distinguished triangle:
  \begin{equation}
    \label{eq:palla}
    \RR q_{2*} (q_{1}^{*}(\UU^{*}(-1)) \otimes \mathbf{U}) \to
    \UU^{*}(-1) \to \pho(\phos(\UU^{*}(-1))).
  \end{equation}

  Since we have $\HH^{k}(X,\UU^{*}(-1))=0$ for all $k$, and
  $\HH^{k}(X,\UU^{*} \ts \UU(-1))=0$ for $k\neq 3$,
  $\hh^{3}(X,\UU^{*} \ts \UU(-1))=1$,
  the lefthandside in \eqref{eq:palla} is isomorphic to $\UU[-2]$.
  Thus we have $\cH^{0}(\pho(\phos(\UU^{*}(-1)))) \cong \UU^{*}(-1)$ and
  $\cH^{1}(\pho(\phos(\UU^{*}(-1)))) \cong \UU$.
  This finishes the proof.
\end{proof}

\subsection{Conics contained in $X$}

In this section we review some facts concerning the geometry 
of conics contained in $X$.
In Proposition \ref{coniche} we 
recover Iliev's description of their Hilbert scheme, see \cite{iliev:sp3}.
We outline a different proof, which holds for any smooth prime Fano threefolds of genus $9$. 

\begin{lem} \label{UC}
  Let $C$ be any conic contained in $X$. Then we have:
  \begin{align}
    \label{eq:UC}
    & \hh^{0}(X,\UU \ts \OO_{C})=1, && \hh^{1}(X,\UU \ts \OO_{C})=0, \\
    \label{eq:UJC}
    & \hom_{X}(\UU,\cI_{C})=1, && \ext^{k}_{X}(\UU, \cI_{C})=0, &&
    \mbox{for $k\neq 1$}
  \end{align}
\end{lem}

\begin{proof}
  By Riemann-Roch we have $\chi(\UU^{*} \ts \cI_{C})=1$, and one can
  easily prove $\Ext^{k}_{X}(\UU, \cI_{C})=0$, for $k \geq 2$.
  So there is at least a
  nonzero global section $s$ of $\UU^{*}$ which vanishes on the curve $C$.
  Note that $s$ lifts to a section $\tilde{s}$ of $\UU^{*}$ on
  $\Sigma$, and $C$ is contained in the vanishing locus of
  $\tilde{s}$.
  This locus is a smooth $3$-dimensional quadric $Q \subset \Sigma$.

  It is easy to see that the restriction of $\UU$ to $Q$ splits as
  $\OO_{Q} \oplus \mathscr{S}$, where $\mathscr{S}$ is the spinor
  bundle on $Q$.
  It is well-known that $\mathscr{S}$ is a stable bundle on $Q$ with
  $\rk(\mathscr{S})=2$ and $c_{1}(\mathscr{S})=-H_{Q}$.
  Moreover, the bundle $\mathscr{S}$ is ACM on $Q$.
  See for instance \cite{ottaviani:spinor}.

  The conic $C$ is the complete intersection of two hyperplanes in
  $Q$, hence we have the Koszul complex:
  \begin{equation}
    \label{eq:CQ}
    0 \to \OO_{Q}(-2\,H_{Q}) \to \OO_{Q}(-H_{Q})^{2} \to \OO_{Q} \to \OO_{C} \to 0.
  \end{equation}

  Tensoring \eqref{eq:CQ} by $\mathscr{S}$, since $\mathscr{S}$ is
  stable and ACM on $Q$, we get $\HH^{k}(C,\mathscr{S})=0$ for all
  $k$.
  This implies \eqref{eq:UC}. Using \eqref{eq:universal-9}, one easily
  gets \eqref{eq:UJC}.
\end{proof}

\begin{lem} \label{lunga}
  Let $F$ be a sheaf in $\Mo_X(2,1,6)$, and let $\alpha$ be any nonzero
  element in $\Hom_{X}(\UU^{*},F)$.
  Then $\alpha$ gives the long exact sequence:
  \begin{equation}
    \label{eq:lunga}
    0 \to \OO_{X} \xr{\beta} \UU^{*} \xr{\alpha} F \to \OO_{C} \to 0,
  \end{equation}
  where $C$ is a conic contained in $X$ and $\beta$ is a global section
  of $\UU^{*}$.
\end{lem}

\begin{proof}
  Let $I$ be the image of a nonzero map
  $\alpha:\UU^* \to F$. Recall by Lemma \ref{stabU} that $\UU$ is stable.
  Thus, by stability of $F$ we get $\rk(\ker\alpha)=1$ and
  $c_1(\ker\alpha)=0$.
  Since $\ker \alpha$ is reflexive, it must be invertible and we get
  an exact sequence of the form:
  \begin{equation}
    \label{eq:toni}
    0 \to \OO_{X} \to \UU^{*} \to I \to 0.
  \end{equation}

  Note that $I$ is easily proved to be stable.
  To get \eqref{eq:lunga}, observe that the cokernel $T$
  of $I \mono F$ satisfies $c_1(T)=0$, $c_2(T)=-2$, $c_3(T)=0$.
  Hence $T$ agrees with $\OO_C$, for some conic
  $C \subset X$, as soon as it has no isolated or embedded points.
  But from \eqref{eq:toni} we get $\HH^1(X,I(-1))=0$ and, since
  $\HH^0(X,F(-1))=0$ by stability, it follows $\HH^0(X,T(-1))=0$ which
  implies our claim.
\end{proof}

\begin{lem} \label{breve}
 Let $F$ be a sheaf in $\Mo_X(2,1,6)$. Then we have:
  \begin{align}
    \nonumber & \hom_{X}(\UU^{*},F)=2, \\
    \label{eq:U*E} & \ext^{k}_{X}(\UU^{*},F)=0, && \mbox{for all $k\geq 1$}.
  \end{align}
\end{lem}

\begin{proof}
  Let us prove \eqref{eq:U*E}.
  For $k=3$, in view of Serre duality, the vanishing of $\Ext^{3}_{X}(\UU^{*},F)$ is easily
  obtained by stability of $\UU^{*}$ and $F$.

  For $k=2$, recall by
\cite[Proposition 3.5]{brambilla-faenzi:genus-7}
that $F$ is a globally generated vector
  bundle, and we have thus an exact sequence:
  \begin{equation}
    \label{eq:E}
    0 \to K \to \OO^{6}_{X} \to F \to 0.
  \end{equation}

  We can prove, as in the proof of \cite[Lemma
  3.3]{brambilla-faenzi:ACM}, that $K$ is a stable vector bundle.
  This gives, since $\UU^*$ is ACM:
  \[
  \Ext^{2}_{X}(\UU^{*},F) \cong \Ext^{3}_{X}(\UU^{*},K) \cong
  \Hom_{X}(K,\UU^{*}(-1))^{*} = 0,
  \]
  where the last vanishing takes place by stability.

  Let us now consider the case $k=1$.
  Observe that $\Hom_{X}(\UU^{*},F) \neq 0$ since
  by Riemann-Roch we have $\chi(\UU^*,F)=2$
  and we have proved \eqref{eq:U*E} for $k=2$.
  A nonzero map $\alpha:\UU^{*} \to F$ must give rise to \eqref{eq:lunga}
  by Lemma \ref{lunga}.
  Tensoring \eqref{eq:lunga} by $\UU$, since $\UU$ is an exceptional
  ACM bundle by Lemma
  \ref{stabU}, we obtain \eqref{eq:U*E} for $k=1$, by virtue of
  \eqref{eq:UC}.
\end{proof}

\begin{lem} \label{E*U}
  Let $F$ be a sheaf in $\Mo_X(2,1,6)$. Then we have
  $\Ext^{k}_{X}(\UU,F^{*})=0$ for all $k$.
\end{lem}

\begin{proof}
Recall that $F$ is a globally generated ACM bundle.
Clearly, we have $\HH^{0}(F^{*} \ts \UU)=0$.
Now, dualize the exact sequence \eqref{eq:E}, and tensor it by $\UU$.
Note that $\mu(K^{*} \ts \UU)=-1/12$, so $\HH^{0}(X,K^{*} \ts \UU) =
\HH^{1}(X,F^{*} \ts \UU) = 0$ by stability.
Similarly, we obtain $\HH^{3}(X,F^{*} \ts \UU)=0$.
By Riemann-Roch we compute $\chi(F^{*} \ts \UU)=0$, so the group
$\HH^{2}(X,F^{*} \ts \UU)$ vanishes too, and our statement is proved.
\end{proof}

\begin{prop}[Iliev] \label{coniche}
  Let $X$ be a smooth prime Fano threefold of genus $9$.
  Then the sheaf $\cV = q_{*}(p^{*}(\UU) \ts \EE)$ is a rank $2$ vector
  bundle on $\Gamma$ with $\deg(\cV)=1$, and we have a natural
  isomorphism:
  \begin{equation}
    \label{eq:V*}
    \cV ^{*} \cong \phs(\UU^{*}).
  \end{equation}

  The Hilbert scheme $\sH^{0}_{2}(X)$ is isomorphic to the projective
  bundle $\p(\cV)$ over $\Gamma$. In particular, $\sH^{0}_{2}(X)$ is a smooth
  irreducible surface.
\end{prop}

\begin{proof}
  In view of Lemma \ref{breve}, we have $\RR^{k} q_{*}(p^{*}(\UU) \ts
  \EE) = 0$, for $k\geq 1$, and $\cV$ is a locally free sheaf on
  $\Gamma$ of rank $\hh^{0}(X,\UU \ts \EE_{y})=2$.

  By an instance of Grothendieck duality, see \cite[Chapter
  III]{hartshorne:residues-duality}, given a sheaf $\sP$ on $X \times
  \Gamma$, we have:
    \begin{equation}
      \label{eq:duality}
      \RRHHom_{\Gamma}(\RR q_{*}(\sP),\OO_{\Gamma}) \cong
      \RR q_{*}  (\OO_{X}(-1) \ts \RRHHom_{X\times \Gamma}(\sP, \OO_{X \times \Gamma}))[3],
    \end{equation}
  and the isomorphism is functorial.
  Setting $\sP = p^{*}(\UU) \ts \EE$ in \eqref{eq:duality}, we get \eqref{eq:V*}.

  Consider now an element $\xi$ of the projective bundle $\p(\cV)$.
  It is uniquely represented by a pair $([\alpha],y)$, where
  $y$ is a point of $\Gamma$, and $[\alpha]$ is an element of
  $\p(\HH^{0}(X,\UU \ts \EE_{y}))$.
  Setting $F=\EE_y$ in Lemma \ref{lunga}, the morphism $\alpha$ gives \eqref{eq:lunga}.
  Applying the functor $\HHom_{X}(-,\OO_{X})$ to \eqref{eq:lunga}, one
  can easily write down the exact sequence:
  \begin{equation}
    \label{eq:non-lunga}
    0 \to \EE_{y}^{*} \xr{\alpha^{\top}} \UU \xr{\beta^{\top}} \cI_{C} \to 0.
  \end{equation}
  This defines an algebraic map $\vartheta: \p(\cV) \mono \sH^{0}_{2}(X)$.

  Let us prove that $\vartheta$ is injective. 
  Consider two elements
  $\xi_{1}=([\alpha_{1}],y_{1})$ and
  $\xi_{2}=([\alpha_{2}],y_{2})$
  and the two conics $C_1,C_2 \subset X$ associated to them. 
  We want to show that $\cI_{C_{1}} \cong \cI_{C_{2}}$, implies
  $\xi_{1}=\xi_{2}$.
  Note that an isomorphism $\gamma : \cI_{C_{1}} \to \cI_{C_{2}}$ lifts to
  a nontrivial map $\tilde{\gamma}: \UU \to \UU$ as soon as:
  \begin{equation}
    \label{eq:UE}
    \Ext^{1}_{X}(\UU,\EE^{*}_{y_{2}})=0,
  \end{equation}
  which in turn is given by Lemma \ref{E*U}.
  Thus, by the simplicity of $\UU$, the map $\tilde{\gamma}$ must be a
  multiple of the identity, and we have an isomorphism
  $\hat{\gamma}:\EE_{y_{1}} \to \EE_{y_{2}}$ with
  $\tilde{\gamma} \circ \alpha_{1}^\top=\alpha_{2}^\top \circ \hat{\gamma}$.
  Then $y_1=y_2$ and $\hat{\gamma}$ is also a multiple of the
  identity. Therefore $\xi_1$ is a multiple of $\xi_2$ and we are done.

  To conclude the proof, let us provide an inverse to $\vartheta$.
  Let $C$ be a conic contained in $X$. By \eqref{eq:UJC}, there exists
  a unique (up to scalar) morphism $\delta : \UU \to \cI_C$.
  Let $I$ be the image of $\delta$ and $K$ its kernel.
  Note that $I$ is torsionfree of rank $1$ so that $K$ is reflexive of
  rank $2$, so $c_3(K)\geq 0$.
  Moreover, by stability of $\UU$ we must have $c_1(I)=0$ and
  $c_1(K)=-1$, and one easily sees that $K$ is stable too.
  Thus $c_2(K)=c_2(K(1)) \geq 6$ by \cite[Lemma
  3.1]{brambilla-faenzi:genus-7}, which easily implies $c_2(I) \leq 2$.
  But $I$ is included in $\cI_C$, hence $c_2(I) \geq 2$, and we get $c_2(I) =
  2$, $c_2(K(1)) = 6$.
  So, \cite[Proposition 3.5]{brambilla-faenzi:genus-7} implies
  $c_3(K)=c_3(I)=0$. We conclude that $I = \cI_C$ (so $\delta$ is
  surjective) and $K = \EE_y^*$ for some $y \in \Gamma$.
  We associate then to the conic $C$ the point $\xi=([\alpha],y)\in \p(\cV)$,
  where $\alpha$ is the transpose of the inclusion of $K$ in
  $\UU$. This provides an algebraic inverse to $\vartheta$ and completes
  the proof.
\end{proof}

\begin{lem}
  We have a natural isomorphism $\phx(\UU(1))[-1] \cong \phs(\UU^{*})$.
  In particular, we get $\det(\cV^{*}) \cong \omega_{\Gamma}(-N)$,
  where $c_{1}(\EE)=H_{X}+N$.
\end{lem}

\begin{proof}
  By Grothendieck duality \eqref{eq:duality}, we get a natural isomorphism:
  \[
  \cV \cong \phs(\UU^{*})^{*} \cong \phx(\UU(1)) \ts
  \omega_{\Gamma}^{*}(N)[-1],
  \]
  and since $\cV$ has rank $2$, the second statement thus follows from the
  first one.

  In view of Proposition \ref{coniche}, the rank $2$ bundle $\phs(\UU^{*})$ is
  stable.
  Thus we only need to show that there is a nonzero morphism from
  $\phx(\UU(1))[-1]$ to $\phs(\UU^{*})$.
  Thus we compute:
  \[
  \begin{split}
  \Hom_{\Gamma}(\phx(\UU(1))[-1],\phs(\UU^{*})) \cong
  \Hom_{X}(\UU(1),\ph(\phs(\UU^{*}))[1]) \cong \Hom_{X}(\UU(1),\UU(1)),
  \end{split}
  \]
  where the last isomorphism follows from Lemma \ref{pallone}.
  This concludes the proof.
\end{proof}

\begin{rmk}
In view of the previous results, we can identify $\cV$ with a
twist of the stable rank $2$ of degree $3$, defined by Iliev
in \cite[Section 5]{iliev:sp3}. Let
$K_{\Gamma}=c_{1}(\omega_{\Gamma})$ and recall that by Mukai's
theorem, \cite{mukai:brill-noether}, $X$ is isomorphic to the {\em type
II Brill-Noether locus}:
\[
\Mo_{\Gamma}(2,K_{\Gamma},3 \cV) = \{\cF \in
\Mo_{\Gamma}(2,c_{1}(\cV)+K_{\Gamma}) \, \lvert \,
\hh^{0}(\Gamma,\cF \ts \cV^{*}) \geq 3 \}.
\]

Therefore, the bundle $\sE$ is universal also for the moduli space $X
\cong \Mo_{\Gamma}(2,K_{\Gamma},3 \cV)$. 
\end{rmk}

\subsection{Lines contained in $X$}

Here we focus on lines contained in $X$.
We describe their Hilbert scheme in the next proposition
as a certain Brill-Noether locus of the Picard variety $\Pic^2(\Gamma)$ of line
bundles of degree $2$ on $\Gamma$.

\begin{prop} \label{brillrette}
  Let $L$ be a line contained in $X$.
  Then we have a functorial exact sequence:
  \begin{equation}
    \label{eq:resL}
    0 \to \OO_{X} \to A_{L} \ts \UU^{*} \xr{\zeta_{L}} \ph(\phx(\OO_{L}(-1))) \to \OO_{L}(-1)
    \to 0,
  \end{equation}
  where $A_{L} = \HH^{1}(L,\UU^{*}(-2))$ has dimension $2$.
  Moreover, the map
  \[
  \psi : L\mapsto \phx(\OO_{L}(-1))
  \]
  gives an isomorphism
  of the Hilbert scheme $\sH^{0}_{1}(X)$ onto a union $W$ of
  components of the locus: 
  \begin{equation}
    \label{eq:WL2}
    \{\cL \in \Pic^{2}(\Gamma) \, \lvert \, \hh^{0}(\Gamma,\cV \ts
    \cL) \geq 2\}. 
  \end{equation}
\end{prop}

\begin{proof}
  Recall that, for each $y\in \Gamma$, the sheaf $\EE_{y}$ is a
  globally generated bundle with $c_{1}(\EE_{y})=1$.
  Thus, it splits over $L$ as $\OO_{L} \oplus \OO_{L}(1)$.
  It follows that $\phx(\OO_{L}(-1))$ is a sheaf concentrated in degree $0$, and
  its rank equals $\hh^{0}(L,\EE_{y}^{*})=1$.
  Its degree is computed by Grothendieck-Riemann-Roch formula.

  To get \eqref{eq:resL}, we use \eqref{triangolo} and \eqref{quadrato}.
  We have thus to compute the cohomology groups
  $\Ext^{k}_{X}(\OO_{L}(-1),\OO_{X})$ and $\Ext^{k}_{X}(\OO_{L}(-1),\UU)$.
  Note that $\UU^{*}$ splits over $L$ as $\OO_{L}^{2} \oplus \OO_{L}(1)$.
  So, using Serre duality, we see that 
  $\Ext^{k}_{X}(\OO_{L}(-1),\OO_{X})=\Ext^{k}_{X}(\OO_{L}(-1),\UU)=0$
  for $k \neq 2$, while for $k=2$ we have $\ext^{2}_{X}(\OO_{L}(-1),\OO_{X})=1$
  and $\ext^{2}_{X}(\OO_{L}(-1),\UU)=2$.
  Setting $A_{L} = \HH^{1}(L,\UU^{*}(-2)) \cong
  \Ext^{2}_{X}(\OO_{L}(-1),\UU)^{*}$, 
  we obtain the functorial resolution \eqref{eq:resL} and
  $\dim(A_L)=2$.

  Set $\cL = \phx(\OO_{L}(-1))$, and recall the isomorphism \eqref{eq:V*}.
  Applying the functor $\Hom_{X}(\UU^{*},-)$ to the long exact
  sequence \eqref{eq:resL}, since $\UU$ is exceptional, and both
  $\Hom_{X}(\UU^{*},\OO_{X})$ and $\Hom_{X}(\UU^{*},\OO_{L}(-1))$ vanish,
  we get a natural isomorphism:
  \[
  \Hom_{\Gamma}(\cV^{*},\cL) \cong \Hom_{X}(\UU^{*},\ph(\cL)) \cong A_{L}.
  \]

  Therefore, the line bundle $\cL$ lies in the locus
  defined by \eqref{eq:WL2}, and actually we have $\hh^{0}(\Gamma,\cV
  \ts \cL) = 2$. 
  Moreover, up to multiplication by a nonzero scalar, the morphism
  $\zeta_{L}$ coincides 
  with the natural evaluation map $e_{\UU^{*},\ph(\cL)}$. 
Thus, the mapping $L\mapsto \phx(\OO_{L}(-1))$ is injective, since
$\OO_L(-1)$ can be recovered as $\cok(\zeta_L)$. 

  Now we shall identify the tangent space of $\sH^{0}_{1}(X)$ at the
  point $[L]$ with
  that of the component $W$, at the point $[\cL]$.
  Note that the morphism $\phs(\zeta_{L})$ must agree with
  the natural evaluation map
  \begin{equation}
    \label{eq:ev}
e_{\cV^*,\cL}:    A_{L} \ts \cV^{*} \to \cL.
  \end{equation}
  
  Remark also that the tangent space to $W$ at the point $[\cL]$ is
  computed as the kernel of the map obtained applying the functor 
  $\Ext^{1}_{\Gamma}(-,\cL)$ to \eqref{eq:ev}.

  Applying the functor $\Hom_{X}(-,\OO_{L}(-1))$ to \eqref{eq:resL},
  and using the obvious vanishing $\HH^{k}(L,\OO_{X}(-1))=0$ for all
  $k$, we obtain a commutative exact diagram:
  \begin{equation}
    \label{eq:diagramma-petri}
    \xymatrix@C+10ex{      
\Ext^{1}_{X}(\ph(\cL),\OO_{L}(-1))
\ar^-{\Ext^{1}(\zeta_{L},\OO_{L}(-1))}[r] \ar[d]^{\cong} & 
      A_{L}^{*} \ts \Ext^{1}_{X}(\UU^{*},\OO_{L}(-1)) \ar[d]^{\cong} \\
       \Ext^{1}_{\Gamma}(\cL,\cL) \ar^-{\Ext^{1}(e_{\cV^*,\cL},\cL)}[r] &
      A_{L}^{*} \ts \Ext^{1}_{\Gamma}(\cV^{*},\cL). 
     }
  \end{equation}

  Here, the kernel (respectively, the cokernel) of
  $\Ext^{1}(\zeta_{L},\OO_{L}(-1))$ is naturally 
  identified with the tangent space $T_{[L]}\sH^{0}_{1}(X) \cong
  \Ext^{1}_{X}(\OO_{L},\OO_{L})$, 
  (respectively, with the obstruction space $\Ext^{2}_{X}(\OO_{L},\OO_{L})$).
  Thus, the diagram \eqref{eq:diagramma-petri} allows to identify the
  tangent space 
  (and the obstruction space) of $\sH^{0}_{1}(X)$ at $[L]$ with
  those of $W$ at $\cL$.
\end{proof}

\begin{rmk} \label{rmk:normal}
  Let $L$ be a line contained in $X$ and set $\cL = \phx(\OO_{L}(-1))$.
  Note that the normal sheaf $\cN_{W}$ at the point $[\cL]$ to the
  subscheme $W$ of $\Pic^2(\Gamma)$  
  is naturally identified with $A_{L}^{*} \ts
  \Ext^{1}_{\Gamma}(\cV^{*},\cL)$, where $A_L$ is canonically
  isomorphic to $\Hom_{\Gamma}(\cV^{*},\cL)$.
  Since $\dim(A_L)=2$ and since $\ext^{1}_{\Gamma}(\cV^{*},\cL) =
  \hh^{1}(L,\UU(-1))=1$, the sheaf $\cN_{W}$ is in fact locally
  free of rank $2$, and its fibre over $[\cL]$ can be identified
  (up to twist by a line bundle on $W$) with $A^*_L$.
\end{rmk}

\begin{rmk} 
  Thanks to the Mukai's interpretation of a line contained in $X$ as
  a locus inside $\Mo_{\Gamma}(2,K_{\Gamma},3 \cV)$, see \cite[Section
  9]{mukai:brill-noether}, it is possible to prove that the map $\psi$
  is actually surjective onto the locus $\{\cL \in \Pic^{2}(\Gamma) \,
  \lvert \, \hh^{0}(\Gamma,\cV \ts \cL) \geq 2\}$. 
\end{rmk}

The following lemmas will be needed further on.

\begin{lem} \label{5}
  Let $L$ be a line contained in $X$. Then we have a natural isomorphism:
  \begin{equation}
    \label{eq:UJL}
    \Hom_{X}(\UU,\cI_{L}) \cong A_{L}^{*}.
  \end{equation}

  The set $S_{L}$ of surjective morphisms $\gamma : \UU \to
  \cI_{L}$ is open and dense in $\p(A_{L})$.
  The subscheme $\p(A_{L}) \setminus S_{L}$ is
  in natural bijection with the length $5$ scheme of reducible
  conics $D\subset X$ which contain $L$.
  For a map $\gamma$ with $[\gamma] \in \p(A_{L}) \setminus S_{L}$,
  we have $\im(\gamma) = \cI_{D}$.
\end{lem}

\begin{proof}
  To get the first statement, we use \eqref{eq:universal-9}
  and we obtain the following natural isomorphisms:
  \[
  \begin{split}
  \Hom_{X}(\UU,\cI_{L}) \cong \HH^{0}(X,\cI_{L} \ts \UU^{*}) \cong
  \HH^{1}(X,\cI_{L} \ts \UU) \cong  
  \HH^{0}(L,\UU) \cong
  \HH^{1}(L,\UU^{*}(-2))^{*} = A_{L}^{*}.
  \end{split}
  \]

  Let now $\gamma$ be a map in $\Hom_{X}(\UU,\cI_{L})$.
  We work as in Proposition \ref{coniche}, and we denote
  $I=\im(\gamma)$, $K=\ker(\gamma)$.
  By stability of $\UU$, the subsheaf $I$ of $\cI_{L}$ has
  trivial determinant.
  Thus $K$ is a reflexive sheaf of rank $2$ with
  $c_{1}=-1$, hence $c_3(K)\geq 0$.
  It is easy to see that $K$ is stable, so
  $c_{2}(K) \geq 6$.
  On the other hand, we have $c_{2}(K) =
  8-c_{2}(I) \leq 7$, so $c_{2}(K)$ equals $6$ or
  $7$.
  If $c_{2}(K)=7$, it follows that $c_{3}(I)\le -1$.
  Then $c_i(I)=c_i(\cI_L)$ for all $i$, so $I=\cI_L$
  i.e. $\gamma$ is surjective.
  We can now assume $c_{2}(K)=6$ and, by \cite[Proposition
  3.5]{brambilla-faenzi:genus-7}, 
  we have that $K$ is a locally free sheaf, so $c_{3}(K)=0$.
  This gives $c_{3}(I)=0$, so $I \cong \cI_{D}$,
  for some conic $D$.
  This proves the last statement.

  Given two non proportional maps $\gamma_{1}$, $\gamma_{2}$ in
  $\Hom_{X}(\UU,\cI_{L})$, assuming that neither is surjective, we get
  $\im(\gamma_{1}) \not \cong \im(\gamma_{2})$ in view of the
  vanishing \eqref{eq:UE}. 
   Therefore, up to a nonzero scalar, each non surjective map $\gamma$
  determines uniquely a conic $D\supset L$.
  The converse is obvious, so it only remains to check that the
  subscheme $\p(A_{L}) \setminus S_{L}$ has length $5$.
  This is true if $L$ is general, see \cite{iskovskih:II}, so we only
  need to check that the length is always finite.
  But $\p(A_{L})$ contains no infinite proper subschemes, so all
  elements $\gamma$ of $\Hom(\UU,\cI_{L})$ should give $\im(\gamma) =
  \cI_{D}$, so $\hom(\UU,\cI_{D})=2$, contradicting Lemma \ref{UC}.
\end{proof}

\begin{lem} \label{lem:OL}
  Let $L$ be a line contained in $X$.
  Then $\phx(\OO_{L})[-1]$ is a line bundle of degree $1$ on the curve
  $\Gamma$.
\end{lem}

\begin{proof}
  Recall that, for each $y\in \Gamma$, the sheaf $\EE_{y}$ is a
  globally generated bundle with $c_{1}(\EE_{y})=1$.
  Thus, it splits over $L$ as $\OO_{L} \oplus \OO_{L}(1)$.
  It follows that $\phx(\OO_{L})$ is a sheaf concentrated in degree $-1$, and
  its rank equals $\hh^{0}(L,\EE_{y}^{*})=1$.
  Its degree is computed by Grothendieck-Riemann-Roch.
\end{proof}

\section{Stable sheaves of rank $2$ with odd determinant}
\label{sec:stable}

Recall from \cite[Theorem 3.12]{brambilla-faenzi:genus-7} that, for each $c_{2} \geq 7$,
there exists a component $\Mo(c_{2})$ of $\Mo_X(2,1,c_{2})$ containing a
locally free sheaf $F$ which satisfies:
\begin{align}
  \label{eq:ipotesona}
  & \HH^{1}(X,F(-1))=0. \\
  \label{eq:liscio}
  & \Ext^{2}_{X}(F,F) = 0,
\end{align}
and the extra assumption $\HH^{0}(X,F \ts
\OO_{L}(-1))=0$, for some line $L\subset X$ having normal bundle
$\OO_{L} \oplus \OO_{L}(-1)$.
For $c_{2}=6$, we have $\Mo(6)=\Mo_X(2,1,6) \cong \Gamma$.
For $c_{2} \geq 7$, $\Mo(c_{2})$ is defined recursively
as the unique component of $\Mo_X(2,1,c_{2})$ which contains a sheaf $F$
fitting into:
\begin{equation} \label{addline}
  0 \to F \to G \to \OO_{L} \to 0,
\end{equation}
where $G$ is a general sheaf lying in $\Mo(c_{2}-1)$.
Here we are going to prove the following result, which amounts to Part
\ref{A} of our main theorem. 

\begin{thm} \label{Main}
  For any integer $c_{2}\geq 7$, there is a birational map $\varphi$,
  generically defined by $F \mapsto \phx(F)$, from $\Mo(c_{2})$ to a generically smooth
  $(2c_{2}-11)$-dimensional component ${\sf B}(c_{2})$ of the locus:
  \begin{equation} 
    \label{eq:brillc2}
    \{\cF \in \Mo_{\Gamma}(c_{2}-6,c_{2}-5) \, \lvert \,
    \hh^{0}({\Gamma},\cV \ts \cF) \geq c_{2}-6 \}.
  \end{equation}
\end{thm}

We begin with a series of lemmas.

\begin{lem} \label{fascio}
  Let $c_{2}\geq 7$, and let $F$ be a sheaf in $\Mo_X(2,1,c_{2})$,
  satisfying \eqref{eq:ipotesona}.
  Then $\phx(F)$ is a vector bundle on $\Gamma$, of rank $c_{2}-6$ and degree $c_{2}-5$.
\end{lem}

\begin{proof}
  Using stability of $F$ and Riemann-Roch's formula we get:
  \begin{equation}
    \label{eq:HF}
    \HH^{k}(X,F(-1))=0, \qquad \mbox{for all $k$}.
  \end{equation}

  By the definition \eqref{phi!} of $\phx$, the stalk of
  $\cH^{k}(\phx(F))$ over the point $y\in \Gamma$ is given by:
  \begin{equation}
    \label{eq:FFy}
    \HH^{k+1}(X,\EE^{*}_{y} \ts F)\ts \omega_{\Gamma,y}.
  \end{equation}

  Let us check that \eqref{eq:FFy} vanishes for all $y\in \Gamma$ and
  for $k \neq 0$.
  For $k=-1$, the statement is clear.
  Indeed, by stability, any
  nonzero morphism $\EE_{y} \to F$ would be an isomorphism for
  $\EE_{y}$ is locally free. But $c_{2}(\EE_{y}) \neq c_{2}(F)$.

  To check the case $k=1$, by Serre duality we can show
  $\Ext^{1}_{X}(F,\EE_{y}) = 0$.
  Setting $E = \EE_{y}$ in \eqref{eq:E}, and applying $\Hom_{X}(F,-)$,
  in view of \eqref{eq:HF} we get:
  \[
  \Ext_{k}^{1}(F,\EE^{*}_{y}) = \Hom_{X}(F,K^{*}) = 0,
  \]
  where the last equality holds by stability.
  Finally, \eqref{eq:FFy} holds for $k=2$ again by stability.

  We have thus proved that $\phx(F)$ is a vector bundle on $\Gamma$.
  By Riemann-Roch we compute its rank as $\rk(\phx(F)) = \chi(F\ts \EE_{y})
  = c_{2} - 6$.
  Using Grothendieck-Riemann-Roch's formula, one can easily compute
  the degree of $\phx(F)$.
\end{proof}

\begin{lem} \label{lem:resolution}
  Let $d\geq 7$, and let $F$ be a sheaf in $\Mo_X(2,1,c_{2})$,
  satisfying \eqref{eq:ipotesona}.
  Then we have a functorial resolution of the form:
  \begin{equation}
    \label{eq:resolution}
    0 \to A_{F} \ts \UU^{*} \xr{\zeta_{F}} \ph(\phx(F)) \to F \to 0,
  \end{equation}
  where $A_{F} = \Ext^{2}_{X}(F,\UU)^{*}$ has dimension $c_{2}-6$.
\end{lem}

\begin{proof}
  To write down \eqref{eq:resolution}, we use the exact triangle \eqref{triangolo}.
  We must calculate the groups
  $\Ext_X^k(F,\OO_X)$ and $\Ext_X^k(F,\UU)$ for all $k$.
  We have proved that the former vanishes for all $k$, see \eqref{eq:HF}.

  If $k=0,3$, we easily get $\Ext_X^k(F,\UU_+)= 0$ by
  stability of the sheaves $\UU_+$ and $\f$.
  Applying the functor $\Hom_X(F,-)$ to
  \eqref{eq:universal-9} we get $\Ext_X^1(F,\UU) \cong \Hom_X(F,\UU^*)=0$,
  where the vanishing follows from the stability of $F$ and $\UU$.
  By Riemann-Roch we get $\ext_X^2(F,\UU)=c_{2}-6$.
\end{proof}

\begin{lem} \label{iniezione}
  Let $c_2\geq 8$, and let $F$ be a sheaf in $\Mo_X(2,1,c_{2})$,
  satisfying \eqref{eq:ipotesona}.
  Then we a natural isomorphism:
  \begin{align}
    \label{eq:nat1} & A_{F} \cong \Hom_{X}(\UU^{*},\ph(\phx(F))), \\
    \label{eq:nat2} & \Ext_{X}^{1}(\UU^{*},F) \cong \Ext_{X}^{1}(\UU^{*},\ph(\phx(F))).
  \end{align}

  In particular, the natural map $\zeta_{F}$ in \eqref{eq:resolution} is uniquely determined up to
  a nonzero scalar.
\end{lem}

\begin{proof}
  In view of Lemma \ref{lem:resolution}, we have the resolution \eqref{eq:resolution}.
  We apply to it the functor $\Hom_{X}(\UU^{*},-)$, and we recall that
  $\UU^{*}$ is exceptional.
  In fact, we are going to show:
  \begin{equation}
    \label{eq:zeri}
    \Ext^{k}_{X}(\UU^{*},F)=0, \qquad \mbox{for $k=0,2,3$},
  \end{equation}
  where the case $k=0$ proves the lemma.
  By contradiction, we consider a nonzero map
  $\gamma:\UU^* \to F$.
  By the argument of Lemma \ref{lunga} we have $\ker (\gamma) \cong
  \OO_{X}$, so $c_{2}(\im(\gamma))=8$, which is impossible for
  $c_{2}(F)\geq 9$. For $c_{2}(F) = 8$, note that
  $c_{3}(\im(\gamma))=2$ gives $c_{2}(\cok(\gamma))=0$,
  $c_{3}(\cok(\gamma))=-2$ which is also impossible.
\end{proof}

Note that it is now immediate to show \eqref{eq:zeri} also for
$k=2,3$. Indeed, for $k\geq 2$, we have:
\[
\Ext^{k}_{X}(\UU^{*},F) \cong \Ext^{k}_{X}(\UU^{*},\ph(\phx(F))) \cong
\Ext^{k}_{\Gamma}(\phs(\UU^{*}),\phx(F)) = 0,
\]
since $\phs(\UU^{*})$ and $\phx(F)$ are
sheaves on a curve.

\begin{lem} \label{lem:petri}
  Let $F$ be a sheaf in $\Mo_X(2,1,c_{2})$ satisfying
  \eqref{eq:ipotesona}, and set $\cF = \phx(F)$.
  Then $\cF$ satisfies $\hh^{0}(\Gamma,\cV \ts \cF) = c_{2}-6$.
  Further, if $F$ satisfies \eqref{eq:liscio},
  then the natural map:
  \begin{equation}
    \label{eq:dualpetri}
    \HH^{0}(\Gamma,\cV \ts \cF) \ts
    \HH^{0}(\Gamma,\cV^{*} \ts \cF \ts \omega_{\Gamma}) \to
    \HH^{0}(\Gamma,\cF^{*} \ts \cF \ts \omega_{\Gamma})
  \end{equation}
  is injective.
\end{lem}

\begin{proof}
  Recall the notation $A_{F}=\Ext^{2}_{X}(F,\UU)^{*}$.
  Note that, by \eqref{eq:nat1},
  \eqref{eq:nat2} and
  \eqref{eq:V*} we have natural isomorphisms:
  \begin{align*}
  & A_{F} \cong \Hom_{\Gamma}(\cV^{*},\cF), \\
  & \Ext_{X}^{1}(\UU^{*},F) \cong \Ext^{1}_{\Gamma}(\cV^{*},\cF),
  \end{align*}
  and we have seen that $A_{F}$ has dimension $c_{2}-6$.

  We have thus proved the first claim,
  and the map $\phs(\zeta_{F})$ must agree up to a nonzero scalar
  with the natural evaluation map:
  \begin{equation}\nonumber
   e:= e_{\cV^*,\cF}:\Hom_{\Gamma}(\cV^{*},\cF) \ts \cV^{*} \to \cF.
  \end{equation}

  We have thus a commutative exact diagram:
    \begin{equation}
    \label{eq:diagramma-petri-2}
    \xymatrix@C+10ex{
      \Ext^{1}_{X}(\ph(\cF),F) \ar^-{\Ext_{X}^{1}(\zeta_{F},F)}[r] \ar[d]^{\cong} &
      A_{F}^{*} \ts \Ext^{1}_{X}(\UU^{*},F) \ar[d]^{\cong} \\
      \Ext^{1}_{\Gamma}(\cF,\cF) \ar^-{\Ext_{\Gamma}^{1}(e,\cF)}[r] &
      A_{F}^{*} \ts \Ext^{1}_{\Gamma}(\cV^{*},\cF).
    }
  \end{equation}

  Therefore, we have the natural isomorphisms:
  \begin{align*}
    & \Ext^{1}_{X}(F,F) \cong \ker (\Ext_{X}^{1}(\zeta_{F},F)) \cong
    \ker (\Ext_{\Gamma}^{1}(e,\cF)), \\
    & \Ext^{2}_{X}(F,F) \cong \cok (\Ext_{X}^{1}(\zeta_{F},F)) \cong
    \cok (\Ext_{\Gamma}^{1}(e,\cF)).
  \end{align*}

  Thus the map $\Ext^1_{\Gamma}(e,\cF)$ is surjective as soon as $F$
  satisfies \eqref{eq:liscio}.
  This implies our claim, since the map \eqref{eq:dualpetri} is the
  transpose of $\Ext^{1}_{\Gamma}(e,\cF)$.
\end{proof}

We are now in position to prove the main result of this section.

\begin{proof}[Proof of Theorem \ref{Main}]
Recall that the variety $\Mo(c_{2})$ contains a vector bundle $F$
satisfying \eqref{eq:ipotesona}, hence by semicontinuity Lemma
\ref{fascio} applies to an open dense subset of $\Mo(c_{2})$.
Thus, for any sheaf $F$ in this open set, $\cF=\phx(F)$ is a vector
bundle on $\Gamma$ of rank $c_{2}-6$ and degree $c_{2}-5$, and it satisfies
$\hh^{0}(\Gamma,\cV \ts \cF)=c_{2}-6$ by by Lemma \ref{lem:petri}.

Let us now prove that, if $F$ is general in $\Mo(c_{2})$, then the vector bundle
$\phx(F)$ is stable over $\Gamma$.
In fact we prove that, if $F$ is a sheaf fitting into
\eqref{addline}, and $G$ is general in $\Mo(c_{2}-1)$, then $\cF = \phx(F)$
is stable over $\Gamma$.
Since stability is an open property by \cite{maruyama:openness}, this will imply that
$\phx(F)$ is stable for $F$ general in $\Mo(c_{2})$.
By induction, we may assume that $\phx(G)$ is a stable vector bundle,
of rank $c_{2}-7$ and degree $c_{2}-6$.

Applying $\phx$ to \eqref{addline}, we get an exact sequence of
bundles on $\Gamma$:
\begin{equation}
  \label{eq:phxaddline}
  0 \to \phx(\OO_{L})[-1] \to \cF \to \phx(G) \to 0,
\end{equation}
where $\phx(\OO_{L})[-1]$ is a line bundle of degree $1$ by Lemma \ref{lem:OL}.
Note that this extension must be nonsplit, for it corresponds to a
nontrivial element in:
\[
\Ext^{1}_{\Gamma}(\phx(G),\phx(\OO_{L})[-1]) \cong \Hom_{X}(G,\OO_{L}).
\]

Assume by contradiction that $\cF$ is not stable, so it contains
a subsheaf $\cK$ with $\mu(\cK) \geq \mu(\cF)$ and $\rk(\cK) <
\rk(\cF)$.
The sequence \eqref{eq:phxaddline}
induces an exact sequence:
\[
0 \to \cK' \to \cK \to \cK'' \to 0,
\]
with $\cK'' \subset \phx(G)$ and $\cK' \subset \phx(\OO_{L}[-1])$.
If $\cK'=0$, then $\mu(\cK) = \mu(\cK'')<\mu(\phx(G))$ for $\phx(G)$ is stable
and \eqref{eq:phxaddline} is nonsplit.
But since $\mu(\phx(G))-\mu(\cF)=\frac{1}{(c_2-5)(c_2-6)}$, we have that
$\mu(\cK)$ cannot fit in the interval $[\mu(\cF),\mu(\phx(G))[$.
If $\rk(\cK')=1$, one can easily apply a similar argument.

We have thus proved that an open dense subset of $\Mo(c_{2})$ maps
into the locus defined by \eqref{eq:brillc2}.
This locus is equipped with
a natural structure of a subvariety of the moduli space
$\Mo_{\Gamma}(c_{2}-6,c_{2}-5)$.
Its tangent space at the point $[\cF]$ is thus
$\ker(\Ext^{1}_{\Gamma}(e,\cF))$, while the obstruction sits in
$\cok(\Ext^{1}_{\Gamma}(e,\cF))$, where again $e=e_{\cV^*,\cF}$. 
Notice that, by Lemma \ref{lem:petri}, the latter space vanishes if $F$
satisfies \eqref{eq:liscio}, so ${\sf B}(c_{2})$ is generically isomorphic
to the $(2d-11)$-dimensional variety $\Mo(c_{2})$.
\end{proof}

\section{The moduli space $\Mo_X(2,1,7)$ as a blowing up of the Picard
variety} \label{sec:7}

In this section, we set up a more detailed study of the moduli space
$\Mo_X(2,1,7)$, of which 
we give a biregular (rather than birational) description. 
In fact, the map $\varphi$ sends the whole space $\Mo_X(2,1,7)$ to
the abelian variety $\Pic^{2}(\Gamma)$.
In turn, $\Pic^{2}(\Gamma)$
contains a copy of the Hilbert scheme
$\sH^{0}_{1}(X)$, via the map $\psi$ (see Proposition \ref{brillrette}), as a subvariety of
codimension $2$.
The relation between these varieties is given by
the main result of this section, which provides Part \ref{B} of our
main theorem.

\begin{thm} \label{main}
  The mapping $\varphi : F \mapsto \phx(F)$ gives an isomorphism of the moduli
  space $\Mo_X(2,1,7)$ to the blowing up of $\Pic^{2}(\Gamma)$ along the
  subvariety $W = \psi(\sH^{0}_{1}(X))$.
  The exceptional divisor consists of the sheaves in $\Mo_X(2,1,7)$
  which are not globally generated.
\end{thm}

We will need some lemmas.

\begin{lem} \label{lemmone}
  Let $F$ be a sheaf in $\Mo_X(2,1,7)$.
  Then, we have:
  \begin{equation}
    \label{eq:veniscing}
    \HH^k(X,F(-1))=\HH^k(X,F)=0, \qquad \mbox{for $k=1,2$.}
  \end{equation}

  Moreover, either $F$ is a locally free, or there exists an exact sequence:
  \begin{equation}\label{doubledual}
    0\to F \to E \to \OO_L \to 0,
  \end{equation}
  where $E$ is a bundle in $\Mo_X(2,1,6)$ and $L$ is a line contained in $X$.

  Furthermore, the following statements are equivalent:
  \begin{enumerate}[i)]
  \item \label{nongg1} the sheaf $F$ is not globally generated;
  \item \label{nongg2} the group $\Hom_X(\UU^*,F)$ is nonzero;
  \item \label{nongg3} there exists a line $L\subset X$, a sheaf $I$ in $\Mo_X(2,1,8,2)$ and two exact sequence:
    \begin{align}
      \label{eq:nongg1} & 0 \to \OO_X \to \UU^* \to I \to 0, \\
      \label{eq:nongg2} & 0 \to I \to F \to \OO_L(-1) \to 0.
    \end{align}
  \end{enumerate}
\end{lem}

\begin{proof}
  The first two statements are taken from \cite[Proposition 3.7]{brambilla-faenzi:genus-7}.
  Clearly condition \eqref{nongg3} implies both conditions \eqref{nongg2} and \eqref{nongg1}.

    Let us prove $\eqref{nongg2} \Rightarrow \eqref{nongg3}$.
  Consider a nonzero map
  $\gamma:\UU^* \to F$.
  The argument of Lemma \ref{lunga} implies $\ker (\gamma) \cong
  \OO_{X}$ and the cokernel $T$
  of $\gamma$ has $c_1(T)=0$, $c_2(T)=-1$, $c_3(T)=-1$,
  so $T \cong \OO_L(-1)$, for some line
  $L \subset X$, if $T$ is supported on a Cohen-Macaulay curve.
  In turn, this holds if the support of $T$ has no isolated or
  embedded points, which follows once we prove $\HH^0(X,T(-1))=0$.
  But \eqref{eq:nongg1} gives $\HH^1(X,I(-1))=0$,
  so by $\HH^0(X,F(-1))=0$, we have $\HH^0(X,T(-1))=0$, and we are done.

  It remains to show $\eqref{nongg1} \Rightarrow \eqref{nongg3}$.
  \begin{description}
  \item[$\eqref{nongg1} \Rightarrow \eqref{eq:nongg2}$]
  Assume that $F$ is not globally generated, that is the evaluation map
  $e_{\OO_X,F}: \HH^0(X,F) \ts \OO_X \to F$ is not surjective. Set
  $K=\ker(e_{\OO_X,F})$, $I=\im(e_{\OO_X,F})$ and $T=\cok(e_{\OO_X,F})$.
  Now it is enough to prove the following facts:
  \begin{align}
    \label{eq:tosee1}  & c_2(T)=-1, \qquad c_3(T)=-1, \\
    \label{eq:tosee2} & \txt{the support of the sheaf $T$ has no isolated or embedded points.}
  \end{align}
  The stability of $F$ easily implies $\rk(I)=2$ and $c_1(I)=1$.
  Since $T$ is a torsion sheaf with $c_1(T)=0$, we have $c_2(T)=-\ell \leq 0$.
  Looking at the sheaf $K$, we see that it is reflexive of rank $3$
  with:
  \[
  c_1(K)=-1,\quad  c_2(K)=9-\ell,\quad c_3(K)=c_3(T)-2+\ell.
  \]

  Thus, we are now reduced to prove $c_3(K)=-2$ and $\ell=1$.
  By Riemann-Roch, we compute $\chi(K)=\frac{1}{2}c_3(K)+1$.
  By definition of the evaluation map $e_{\OO_X,F}$, taking global sections of the composition:
  \[
  \HH^0(X,F)\otimes \OO_X \epi I \mono F,
  \]
  we obtain an isomorphism. This implies:
  \begin{align}
  &\HH^0(X,K)=\HH^1(X,K)=0, \label{tuttozero}\\
  &\HH^0(X,T)\cong \HH^1(X,I) \cong \HH^2(X,K), \label{giacomo}
  \end{align}
  and one can easily see $\HH^3(X,K)=0$.

  We postpone the proof of the following claim, and we assume it for
  the time being.
  \begin{claim} \label{cappa}
    We have $c_2(K)\in \{8,9 \}$ and $\HH^2(X,K) = 0$.
  \end{claim}

  Note that the second statement of the above claim proves $\chi(K)=0$, which implies $c_3(K)=-2$.
  Then, by the first statement of Claim \ref{cappa}, we obtain $\ell=1$, for
  otherwise $T$ would be zero. This proves \eqref{eq:tosee1}.
  Note that \eqref{eq:tosee2} follows from the vanishing of
  \eqref{giacomo}. This finishes the proof.

  \item[$\eqref{nongg1} \Rightarrow \eqref{eq:nongg1}$]
    Note that $\chi(\UU^*,K)=-1$.
    Since $\Ext^3_X(\UU^*,K)=0$ by stability, we get
    $\Ext^1_X(\UU^*,K) \neq 0$.
    Applying the functor $\Hom(\UU^*,-)$ to the sequence:
    \[
    0\to K \to \HH^0(X,F)\otimes \OO_X \to I \to 0,
    \]
    one easily obtains an isomorphism  $\Hom_X(\UU^*,I)\cong
    \Ext^1_X(\UU^*,K)$. Then there exists a nontrivial morphism
    $\alpha : \UU^* \to I$.
    The calculation used to prove $\eqref{nongg2} \Rightarrow \eqref{nongg3}$
    shows that $\alpha$ is surjective, so we get \eqref{eq:nongg1}.
 \end{description}
\end{proof}

\begin{proof}[Proof of Claim \ref{cappa}]
  We observe that the restriction of $K$ to a general hyperplane
  section $S$ is stable, using Hoppe's criterion.
  Indeed, we have $\HH^0(S,K)=0$ by \eqref{tuttozero}, while
  the group $\HH^0(S,\wedge^2 K)$ vanishes since it is a subgroup of
  $\HH^0(S, K)\otimes\HH^0(S,F)=0$.
  Then from \eqref{eq:dimension} it follows that $c_2(K)\geq 8$. This
  proves the first assertion.

  Let us now show the second one.
  Tensoring \eqref{eq:sectionS} by
  $K(1)$, we are reduced to show the vanishing of the groups
  $\HH^2(X,K(1))$ and $\HH^1(S,K(1))$.

  Looking at the first one, assume by contradiction that there exists a nontrivial extension of
  the form:
  \[
    0\to\OO_X(-1) \to \tilde K \to K(1) \to 0,
  \]
  where $\tilde K$ is a rank $4$ vector bundle with $c_1(\tilde K)=1$
  and $c_2(\tilde K)<0$.
  Then $\tilde K$ is not semistable by
  Bogomolov's inequality \eqref{eq:bogomolov}.
  By considering the possible values of the slope of a
  destabilizing subsheaf of $\tilde{K}$, one sees that
  Harder-Narasimhan filtration has the form $0\subset K_1 \subset
  \tilde K$ and $Q=\tilde K/K_1$ is semistable, and $\mu(K_1)$ can be
  either $\frac{1}{2}$ or $\frac{1}{3}$.
  Then by Bogomolov's inequality we have $c_2(K_1)\geq0$.
  In any case $c_1(Q)=0$, so $c_2(Q)\geq 0$. This contradicts
  $c_2(\tilde K)<0$.

  Let us now turn to the group $\HH^1(S,K(1))$, and observe that it is
  dual to $\Ext^1_S(K_S(1),\OO_S)$.
  Assuming it to be nontrivial, we get a nonsplit exact sequence on $S$ of the form:
  \begin{equation}\label{eq:cappaS}
    0\to\OO_S \to \widetilde{K_S} \to K_S(1) \to 0,
  \end{equation}
  where $\widetilde{K_S}$ is a rank $4$ vector bundle on $S$ with $c_1(\widetilde{K_S})=2$
  and $c_2(\widetilde{K_S})\leq25$.
  Then $\widetilde{K_S}$ is not stable
  by \eqref{eq:dimension}.
  This time one can check that the only
  possible destabilizing subsheaf $K_1$ must have slope
  $\frac{1}{2}$. The same happens to $Q=\widetilde{K_S}/K_1$.
  By semistability of $K_1$ and $Q$ one has
  \[
  c_2(\widetilde{K_S})=c_2(K_1)+c_2(Q)+16 \geq 28,
  \]
  a contradiction.
\end{proof}

\begin{lem} \label{lem:1}
  The map $\varphi : F \to \phx(F)$ sends $\Mo_X(2,1,7)$ to $\Pic^{2}(\Gamma)$.
  If the sheaf $F$ is globally generated, then $\varphi$ is a local isomorphism
  around $F$.
\end{lem}

\begin{proof}
  Set $\cF = \phx(F)$.
  In view of \eqref{eq:veniscing} and Lemma \ref{fascio}, the map
  $\varphi$ takes values in $\Pic^{2}(\Gamma)$.
  Assume now $F$ globally generated. By Lemma \ref{lemmone} we have
  $\Hom_{X}(\UU^{*},F)=0$, so by applying the functor
  $\Hom_X(\UU^*,-)$ to the resolution \eqref{eq:resolution}
  we get \eqref{eq:nat1}, from which it follows that
  $\varphi$ is injective at $F$.

  Recall that $\Ext^{k}_{X}(\UU^{*},F)=0$ for $k=2,3$, and by
  Riemann-Roch we have $\chi(\UU^{*},F)=0$.
  Thus we must also have:
  \[
  \Ext^{1}_{X}(\UU^{*},F)=0,
  \]
  so $\Ext^{1}_{\Gamma}(\cV^*,\cF) = 0$.
  Therefore the map $\Ext^{1}_{\Gamma}(e_{\cV^*,\cF},\cF)$ is zero.
  Now, by the infinitesimal analysis of Lemma \ref{lem:petri}, the
  differential of $\varphi$ at $[F]$ induces an isomorphism:
  \[
  \Ext^{1}_{X}(F,F) \cong \ker (\Ext^{1}_{\Gamma}(e_{\cV^*,\cF},\cF))
  = \Ext^{1}_{\Gamma}(\cF,\cF) \cong \HH^{1}(\Gamma,\OO_{\Gamma}).
  \]
  
\end{proof}

Recall that we denote by $A_{L}$ the $2$ dimensional vector space $\Hom_{X}(\UU,\cI_{L})^{*}$.

\begin{lem} \label{sugiu}
  Let $L$ be a line contained in $X$.
  Then there is a natural injective map $\eta : \p(A_{L}) \mapsto \Mo_X(2,1,7)$
  such that any sheaf $F$ in the image of $\eta$ sits into \eqref{eq:nongg2}, for
  some sheaf $I$ sitting in \eqref{eq:nongg1}.
\end{lem}

\begin{proof}
  Let us define the map $\eta: \p(A_{L}) \to \Mo_X(2,1,7)$.
  In view of Lemma \ref{5}, for any element $[\gamma] \in \p(A_{L})$, we have two
  alternatives:
  \begin{enumerate}[i)]
  \item \label{su} the map $\gamma$ is surjective;
  \item \label{giu} the image of the map $\gamma$ is isomorphic to $\cI_{C}$, for
    some reducible conic $C \subset X$ which is the union of $L$ and
    another line $L' \subset X$.
  \end{enumerate}

  If \eqref{su} takes place, we define $\eta([\gamma])$ as the dual
  of $\ker(\gamma)$. This sheaf is easily seen to lie in $\Mo_X(2,1,7)$.
  Note that this correspondence is one to one.
  Indeed, assuming $\eta([\gamma_{1}]) = \eta([\gamma_{2}])$, we
  would have $K_{1} = \ker(\gamma_{1}) \cong K_{2} = \ker(\gamma_{2})$.
  But the isomorphism $K_{1} \cong K_{2}$ would then lift to an
  isomorphism $\UU \to \UU$, for $\Ext^{1}_{X}(\cI_{L},\UU)=0$,
  indeed:
  \[
  \ext^{1}_{X}(\cI_{L},\UU)= \ext^{2}_{X}(\OO_{L},\UU)=
  \ext^{1}_{X}(\UU,\OO_L(-1))= \hh^1(L,\UU^*(-1))=0.
  \]
  Since both $\UU$ and $\cI_{L}$ are simple, this would then mean that
  $\gamma_{1}$ is a multiple of $\gamma_{2}$.

  Assume now that \eqref{giu} takes place.
  We have thus an exact sequence of the form \eqref{eq:non-lunga},
  with $\beta^{\top} = \gamma$.
  Since $C$ contains $L$, we have:
  \begin{equation}
    \label{eq:C}
    0 \to \OO_{L}(-1) \to \OO_{C} \to \OO_{L'} \to 0,
  \end{equation}
  for some line $L'\subset X$.
  Dualizing \eqref{eq:non-lunga} one obtains \eqref{eq:lunga}.
  Patching this with \eqref{eq:C},
  we define a surjective map as the composition $\ker(\gamma)^{*} \epi \OO_{C} \epi
  \OO_{L'}$, and we let $\eta([\gamma])$ be the kernel of this map.
  Again, one proves easily that this sheaf lies in $\Mo_X(2,1,7)$.

  We prove now that $\eta$ is injective also in this case.
  Let us take two maps $\gamma_1$, $\gamma_2$ in $\p(A_L)$ and let $F$
  be the sheaf representing 
  $\eta(\gamma_1)=\eta(\gamma_2)$.
  Let $E=F^{**}$. We have $E/F = \OO_{L'}$, for some line $L' \subset
  X$, and $E$ lies in 
  $\Mo_X(2,1,6)$.
  By construction $\ker(\gamma_i) \cong E^*$, and $\im(\gamma_i)=\cI_C$ for all $i$, 
  where the conic $C$ is $L \cup L'$.
  Since $\Ext^1_X(\UU,E^*)=0$ by Lemma \ref{E*U}, and since $\UU$ is a
  simple sheaf, we conclude that $\gamma_1$ is proportional to $\gamma_2$.

  Finally, it is clear by the definition that in both cases \eqref{su}
  and \eqref{giu}, the sheaf defined by $\eta([\gamma])$ sits into \eqref{eq:nongg2}.
\end{proof}

\begin{lem} \label{sugiu2}
  Let $G$ be a sheaf in $\Mo_X(2,1,7)$, and assume that $G$ is not
  globally generated.
  Then the set of sheaves $F$ in $\Mo_X(2,1,7)$ such that
  $\varphi(F)=\varphi(G)$ is identified with
  $\eta(\p(A_{L}))$, for some line $L\subset X$.

  The subscheme of those sheaves $F$ which are not locally free,
  and satisfy $\varphi(F)=\varphi(G)$, has length $5$.
\end{lem}

\begin{proof}
  In view of Lemma \ref{lemmone}, there exists a line $L\subset X$
  such that $G$ is not globally generated over $L$, i.e. we have the exact
  sequence \eqref{eq:nongg2}, with $F$ replaced by $G$.
  Applying the functor $\phx$ to this exact sequence, we get:
  \begin{equation}
    \label{eq:psiphi}
    \varphi(G) = \phx(G) \cong \phx(\OO_{L}(-1)) = \psi([L]),
  \end{equation}
  where $\psi$ is given by Proposition \ref{brillrette}.

  Since $\varphi$ is a local isomorphism on the set of globally
  generated sheaves, any sheaf $F$ with $\varphi(F)=\varphi(G)$ must
  not be globally generated.
  Dualizing \eqref{eq:nongg2} and \eqref{eq:nongg1} we obtain $F^{*}
  \cong I^{*}$ and:
  \begin{align}
    \nonumber & 0 \to F^{*} \to \UU \xr{\delta} \OO_{X} \to \EExt^{1}_{X}(I,\OO_{X}) \to 0 \\
    \label{eq:I*} & 0 \to \EExt^{1}_{X}(F,\OO_{X}) \to \EExt^{1}_{X}(I,\OO_{X}) \to
    \OO_{L} \to 0.
  \end{align}

  We have here the following two alternatives.
  \begin{enumerate}[a)]
  \item \label{a} the sheaf $F$ is locally free, and $\im(\delta) \cong \cI_{L}$;
  \item \label{b} we have $F/F^{**} \cong \OO_{L'}$ for some line $L'\subset X$,
    and by \eqref{doubledual} this implies:
    \[
    F^{**} \in \Mo_X(2,1,6), \qquad \im(\delta) \cong \cI_{C},
    \]
    where the conic $C$ is $L \cup L'$, and \eqref{eq:I*} becomes
    of the form \eqref{eq:C}.
  \end{enumerate}

  We let $\gamma$ be the
  restriction of $\delta$ to its image $\cI_{L}$.
  Clearly, if \eqref{a} takes place, then $F$ is isomorphic to
  $\eta([\gamma])$, and $\gamma$ is as in case \eqref{su} of Lemma
  \ref{sugiu}.

  Similarly, if \eqref{b} takes place, then $\gamma$ is as in case \eqref{giu} of Lemma
  \ref{sugiu}, and $F$ is isomorphic to $\eta([\gamma])$.
  The set of sheaves $F$ which are not locally free and with $\varphi(F) = \varphi(G)$ is thus in natural
  bijection with the set of elements $[\gamma]$ in $\p(A_{L})$ such
  that $\gamma$ is not surjective.
  By Lemma \ref{5}, this is identified with the set of reducible
  conics which contain $L$, which has length $5$.
\end{proof}

We are now in position to prove our main result.

\begin{proof}[Proof of Theorem \ref{main}]
  We have seen in Lemma \ref{lem:1} that $\varphi$ is a local
  isomorphism along the open set of globally generated sheaves.

  On the other hand, the map $\varphi$ equips the subscheme of
  sheaves which are not globally generated with a structure of $\p^{1}$
  bundle over $W=\psi(\sH^{0}_{1}(X))$.
  Indeed, if a sheaf $G$ is not globally generated over a line $L$, by
  \eqref{eq:psiphi}, $\varphi(G)=\psi([L])$ lies in $W$.
  Moreover by Lemmas \ref{sugiu} and \ref{sugiu2}, we have
  $\varphi(\eta([\gamma])) = \psi([L])$, for all $[\gamma] \in \p(A_L)$.

  Thus, it only remains to provide a natural identification of the
  fibre of $\varphi(G)$ with the projectivized normal bundle
  of $W \subset \Pic^{2}(\Gamma)$ at the point $\psi([L])$.
  By Remark \ref{rmk:normal} and Proposition \ref{brillrette}, the latter is
  canonically identified with $\p(A_{L})$.
  On the other hand, by Lemmas \ref{sugiu} and \ref{sugiu2}, via the
  map $\eta$ the former is also naturally identified with
  $\p(A_{L})$.
  This concludes the proof.
\end{proof}


\begin{thebibliography}{AHDM78}


\bibitem[AO94]{ancona-ottaviani:special}
{\sc Vincenzo Ancona and Giorgio Ottaviani}, \emph{Stability of special
  instanton bundles on {${\bf P}\sp {2n+1}$}}, Trans. Amer. Math. Soc.
  \textbf{341} (1994), no.~2, 677--693.

\bibitem[AF06]{enrique-dani:v22}
{\sc Enrique Arrondo and Daniele Faenzi}, \emph{Vector bundles with no
  intermediate cohomology on {F}ano threefolds of type {$V\sb {22}$}}, Pacific
  J. Math. \textbf{225} (2006), no.~2, 201--220.

\bibitem[AHDM78]{adhm}
{\sc Michael~F. Atiyah, Nigel~J. Hitchin, Vladimir~G. Drinfel{\cprime}d, and
  Yuri~I. Manin}, \emph{Construction of instantons}, Phys. Lett. A \textbf{65}
  (1978), no.~3, 185--187.


\bibitem[AW77]{atiyah-ward}
{\sc Michael~F. Atiyah and Richard~S. Ward}, \emph{Instantons and algebraic geometry},
  Comm. Math. Phys. \textbf{55} (1977), no.~2, 117--124.

\bibitem[Bar77]{barth:some-properties}
{\sc Wolf Barth}, \emph{Some properties of stable rank-{$2$} vector bundles on
  {${\bf P}\sb{n}$}}, Math. Ann. \textbf{226} (1977), no.~2, 125--150.

\bibitem[BH78]{barth-hulek}
{\sc Wolf Barth and Klaus Hulek}, \emph{Monads and moduli of vector bundles},
  Manuscripta Math. \textbf{25} (1978), no.~4, 323--347.

\bibitem[BF08a]{brambilla-faenzi:genus-7}
{\sc Maria~Chiara Brambilla and Daniele Faenzi}, \emph{{V}ector bundles on
  {F}ano threefolds of genus 7 and {B}rill-{N}oether loci}, Arxiv preprint, {\tt http://arxiv.org/abs/0810.3138}, 2008.

\bibitem[BF08b]{brambilla-faenzi:ACM}
{\sc \bysame}, \emph{Moduli spaces of rank 2 {ACM} bundles on prime {F}ano
  threefolds}, Arxiv preprint, {\tt http://arxiv.org/abs/0806.2265}, 2008.


\bibitem[Bri99]{bridgeland:fourier-mukai}
{\sc Tom Bridgeland}, \emph{Equivalences of triangulated categories and
  {F}ourier-{M}ukai transforms}, Bull. London Math. Soc. \textbf{31} (1999),
  no.~1, 25--34.

\bibitem[CTT03]{coanda-tikhomirov-trautmann}
{\sc Iustin Coand{\u{a}}, Alexander~S. Tikhomirov, and G{\"u}nther Trautmann},
  \emph{Irreducibility and smoothness of the moduli space of mathematical
  5-instantons over {$\Bbb P\sb 3$}}, Internat. J. Math. \textbf{14} (2003),
  no.~1, 29--53.

\bibitem[Dru00]{druel:cubic-3-fold}
{\sc St{\'e}phane Druel}, \emph{Espace des modules des faisceaux de rang 2
  semi-stables de classes de {C}hern {$c\sb 1=0,\ c\sb 2=2$} et {$c\sb 3=0$}
  sur la cubique de {${\bf P}\sp 4$}}, Internat. Math. Res. Notices (2000),
  no.~19, 985--1004.

\bibitem[GM96]{gelfand-manin:homological}
{\sc Sergei~I. Gelfand and Yuri~I. Manin}, \emph{Methods of homological
  algebra}, Springer-Verlag, Berlin, 1996, Translated from the 1988 Russian
  original.

\bibitem[GLN06]{gruson-laytimi-nagaraj}
{\sc Laurent Gruson, Fatima Laytimi, and Donihakkalu~S. Nagaraj}, \emph{On
  prime {F}ano threefolds of genus 9}, Internat. J. Math. \textbf{17} (2006),
  no.~3, 253--261.

\bibitem[Har66]{hartshorne:residues-duality}
{\sc Robin Hartshorne}, \emph{Residues and duality}, Lecture notes of a seminar
  on the work of A. Grothendieck, given at Harvard 1963/64. With an appendix by
  P. Deligne. Lecture Notes in Mathematics, No. 20, Springer-Verlag, Berlin,
  1966.

\bibitem[Hop84]{hoppe:rang-4}
{\sc Hans~J{\"u}rgen Hoppe}, \emph{Generischer {S}paltungstyp und zweite
  {C}hernklasse stabiler {V}ektorraumb\"undel vom {R}ang {$4$} auf {${\bf
  P}\sb{4}$}}, Math. Z. \textbf{187} (1984), no.~3, 345--360.

\bibitem[HL97]{huybrechts-lehn:moduli}
{\sc Daniel Huybrechts and Manfred Lehn}, \emph{The geometry of moduli spaces
  of sheaves}, Aspects of Mathematics, E31, Friedr. Vieweg \& Sohn,
  Braunschweig, 1997.

\bibitem[Ili03]{iliev:sp3}
{\sc Atanas Iliev}, \emph{The {${\rm Sp}\sb 3$}-{G}rassmannian and duality for
  prime {F}ano threefolds of genus 9}, Manuscripta Math. \textbf{112} (2003),
  no.~1, 29--53.

\bibitem[IM07a]{iliev-manivel:genus-8}
{\sc Atanas Iliev and Laurent Manivel}, \emph{Pfaffian lines and vector bundles
  on {F}ano threefolds of genus 8}, J. Algebraic Geom. \textbf{16} (2007),
  no.~3, 499--530.

\bibitem[IM00]{iliev-markushevich:cubic}
{\sc Atanas Iliev and Dimitri Markushevich}, \emph{The {A}bel-{J}acobi map for
  a cubic threefold and periods of {F}ano threefolds of degree 14}, Doc. Math.
  \textbf{5} (2000), 23--47 (electronic).

\bibitem[IM04a]{iliev-markushevich:genus-7}
{\sc \bysame}, \emph{Elliptic curves and rank-2
  vector bundles on the prime {F}ano threefold of genus 7}, Adv. Geom.
  \textbf{4} (2004), no.~3, 287--318.

\bibitem[IM07b]{iliev-markushevich:sing-theta:asian}
{\sc \bysame}, \emph{Parametrization of {S}ing
  {$\Theta$} for a {F}ano 3-fold of genus 7 by moduli of vector bundles}, Asian
  J. Math. \textbf{11} (2007), no.~3, 427--458.

\bibitem[IR05]{iliev-ranestad}
{\sc Atanas Iliev and Kristian Ranestad}, \emph{Geometry of the {L}agrangian
  {G}rassmannian {${\bf LG}(3,6)$} with applications to {B}rill-{N}oether
  loci}, Michigan Math. J. \textbf{53} (2005), no.~2, 383--417.

\bibitem[Isk78]{iskovskih:II}
{\sc Vasilii~A. Iskovskih}, \emph{Fano threefolds. {II}}, Izv. Akad. Nauk SSSR
  Ser. Mat. \textbf{42} (1978), no.~3, 506--549, English translation in Math.
  U.S.S.R. Izvestija \textbf{12} (1978) no. 3 , 469--506 (translated by Miles
  Reid).

\bibitem[IP99]{fano-encyclo}
{\sc Vasilii~A. Iskovskikh and Yuri.~G. Prokhorov}, \emph{Fano varieties},
  Algebraic geometry, V, Encyclopaedia Math. Sci., vol.~47, Springer, Berlin,
  1999, pp.~1--247.

\bibitem[KO03]{kastylo-ottaviani}
{\sc Pavel~I. Katsylo and Giorgio Ottaviani}, \emph{Regularity of the moduli
  space of instanton bundles {${\rm MI}\sb {{\bf P}\sp 3}(5)$}}, Transform.
  Groups \textbf{8} (2003), no.~2, 147--158.

\bibitem[Kuz96]{kuznetsov:V22}
{\sc Alexander~G. Kuznetsov}, \emph{An exceptional set of vector bundles on the
  varieties {$V\sb {22}$}}, Vestnik Moskov. Univ. Ser. I Mat. Mekh. (1996),
  no.~3, 41--44, 92.

\bibitem[Kuz05]{kuznetsov:v12}
{\sc \bysame}, \emph{Derived categories of the {F}ano threefolds {$V\sb
  {12}$}}, Mat. Zametki \textbf{78} (2005), no.~4, 579--594, English
  translation in {M}ath. {N}otes {\bf 78}, no. 3-4, 537--550 (2005).

\bibitem[Kuz06]{kuznetsov:hyperplane}
{\sc \bysame}, \emph{Hyperplane sections and derived categories}, Izv. Ross.
  Akad. Nauk Ser. Mat. \textbf{70} (2006), no.~3, 23--128.

\bibitem[MT01]{markushevich-tikhomirov}
{\sc Dimitri Markushevich and Alexander~S. Tikhomirov}, \emph{The
  {A}bel-{J}acobi map of a moduli component of vector bundles on the cubic
  threefold}, J. Algebraic Geom. \textbf{10} (2001), no.~1, 37--62.

\bibitem[Mar76]{maruyama:openness}
{\sc Masaki Maruyama}, \emph{Openness of a family of torsion free sheaves}, J.
  Math. Kyoto Univ. \textbf{16} (1976), no.~3, 627--637.

\bibitem[Mar80]{maruyama:boundedness-small}
{\sc \bysame}, \emph{Boundedness of semistable sheaves of small ranks}, Nagoya
  Math. J. \textbf{78} (1980), 65--94.

\bibitem[Mar81]{maruyama:boundedness}
{\sc \bysame}, \emph{On boundedness of families of torsion free sheaves}, J.
  Math. Kyoto Univ. \textbf{21} (1981), no.~4, 673--701.

\bibitem[Mer01]{mercat:pente-2}
{\sc Vincent Mercat}, \emph{Fibr\'es stables de pente 2}, Bull. London Math.
  Soc. \textbf{33} (2001), no.~5, 535--542.

\bibitem[Muk84]{mukai:symplectic}
{\sc Shigeru Mukai}, \emph{Symplectic structure of the moduli space of sheaves
  on an abelian or {$K3$} surface}, Invent. Math. \textbf{77} (1984), no.~1,
  101--116.

\bibitem[Muk88]{mukai:curves-K3}
{\sc \bysame}, \emph{Curves, {$K3$} surfaces and {F}ano {$3$}-folds of genus
  {$\leq 10$}}, Algebraic geometry and commutative algebra, Vol.\ I,
  Kinokuniya, Tokyo, 1988, pp.~357--377.

\bibitem[Muk89]{mukai:biregular}
{\sc \bysame}, \emph{Biregular classification of {F}ano {$3$}-folds and {F}ano
  manifolds of coindex {$3$}}, Proc. Nat. Acad. Sci. U.S.A. \textbf{86} (1989),
  no.~9, 3000--3002.

\bibitem[Muk95]{mukai:curves-symmetric-I}
{\sc \bysame}, \emph{Curves and symmetric spaces. {I}}, Amer. J. Math.
  \textbf{117} (1995), no.~6, 1627--1644.

\bibitem[Muk01]{mukai:brill-noether}
{\sc \bysame}, \emph{Non-abelian {B}rill-{N}oether theory and {F}ano 3-folds
  [translation of {S}\=ugaku {\bf 49} (1997), no. 1, 1--24; {MR} 99b:14012]},
  Sugaku Expositions \textbf{14} (2001), no.~2, 125--153, Sugaku Expositions.

\bibitem[OSS80]{okonek-schneider-spindler}
{\sc Christian Okonek, Michael Schneider, and Heinz Spindler}, \emph{Vector
  bundles on complex projective spaces}, Progress in Mathematics, vol.~3,
  Birkh\"auser Boston, Mass., 1980.

\bibitem[Ott88]{ottaviani:spinor}
{\sc Giorgio Ottaviani}, \emph{Spinor bundles on quadrics}, Trans. Amer. Math.
  Soc. \textbf{307} (1988), no.~1, 301--316.

\bibitem[TiB91a]{teixidor-i-bigas:brill-noether-stable}
{\sc Montserrat Teixidor~i Bigas}, \emph{Brill-{N}oether theory for stable
  vector bundles}, Duke Math. J. \textbf{62} (1991), no.~2, 385--400.

\bibitem[TiB91b]{teixidor-i-bigas:rank-2}
{\sc \bysame}, \emph{Brill-{N}oether theory for vector bundles of rank {$2$}},
  Tohoku Math. J. (2) \textbf{43} (1991), no.~1, 123--126.

\bibitem[Wei94]{weibel:homological}
{\sc Charles~A. Weibel}, \emph{An introduction to homological algebra},
  Cambridge Studies in Advanced Mathematics, vol.~38, Cambridge University
  Press, Cambridge, 1994.

\end{thebibliography}

\def\cprime{$'$} \def\cprime{$'$} \def\cprime{$'$} \def\cprime{$'$}
  \def\cprime{$'$} \def\cprime{$'$}
\providecommand{\bysame}{\leavevmode\hbox to3em{\hrulefill}\thinspace}
\providecommand{\MR}{\relax\ifhmode\unskip\space\fi MR }
\providecommand{\MRhref}[2]{%
  \href{http://www.ams.org/mathscinet-getitem?mr=#1}{#2}
}
\providecommand{\href}[2]{#2}

\end{document}